\documentclass{gOMS2e}

\usepackage{epstopdf}% To incorporate .eps illustrations using PDFLaTeX, etc.
\usepackage{subfigure}% Support for small, `sub' figures and tables
 \usepackage{hyperref}
\hypersetup{colorlinks=true}

\theoremstyle{plain}% Theorem-like structures
\newtheorem{theorem}{Theorem}[section]

\newtheorem{lemma}[theorem]{Lemma}

\numberwithin{equation}{section}

 \usepackage{algorithm}
\usepackage{algpseudocode}
\usepackage{wrapfig}
\usepackage{lscape}

\theoremstyle{definition}

\theoremstyle{remark}

\newcommand{\defeq}{\stackrel{\rm def}{=}}

\usepackage{xargs}                      % Use more than one optional parameter in a new commands
\usepackage[pdftex,dvipsnames]{xcolor}  % Coloured text etc.
\usepackage[colorinlistoftodos,prependcaption,textsize=tiny]{todonotes}
\newcommandx{\unsure}[2][1=]{\todo[linecolor=red,backgroundcolor=red!25,bordercolor=red,#1]{#2}}
\newcommandx{\change}[2][1=]{\todo[linecolor=blue,backgroundcolor=blue!25,bordercolor=blue,#1]{#2}}
\newcommandx{\info}[2][1=]{\todo[linecolor=OliveGreen,backgroundcolor=OliveGreen!25,bordercolor=OliveGreen,#1]{#2}}
\newcommandx{\improvement}[2][1=]{\todo[linecolor=Plum,backgroundcolor=Plum!25,bordercolor=Plum,#1]{#2}}
\newcommandx{\thiswillnotshow}[2][1=]{\todo[disable,#1]{#2}}

\usepackage{booktabs}
\usepackage{array}
\usepackage{arydshln}
\begin{document}

%\jvol{00} \jnum{00} \jyear{2015} \jmonth{July}

%\articletype{GUIDE}

\title{A Robust Multi-Batch L-BFGS Method for Machine Learning\footnote{This work substantially extends \cite{Berahas2016} published at the Neural Information Processing Systems (NeurIPS) conference in 2016.}}

%\maketitle

\author{
\name{Albert S. Berahas\textsuperscript{a}$^{\ast}$\thanks{$^\ast$Corresponding author. Email: \url{albertberahas@u.northwestern.edu}}
and Martin Tak\'a\v c\textsuperscript{a}}
\affil{\textsuperscript{a}Department of Industrial and Systems Engineering, Lehigh University, Bethlehem, PA, USA}
\received{v5.0 released July 2015}
}

\maketitle

\begin{abstract}
This paper describes an implementation of the L-BFGS method designed to deal with two adversarial situations. The first occurs in distributed computing environments where some of the computational nodes devoted to the evaluation of the function and gradient are unable to return results on time. A similar challenge occurs in a multi-batch approach in which the data points used to compute function and gradients are purposely changed at each iteration to accelerate the learning process. Difficulties arise because L-BFGS employs gradient differences to update the Hessian approximations, and when these gradients are computed using different data points the updating process can be unstable. This paper shows how to perform stable quasi-Newton updating in the multi-batch setting, studies the convergence properties for both convex and nonconvex functions, and illustrates the behavior of the algorithm in a distributed computing platform on binary classification logistic regression and neural network training problems that arise in machine learning.
\end{abstract}


\begin{classcode}90C30; 90C06; 90C53. \end{classcode}

\section{Introduction}
\label{sec:intro}
\setcounter{equation}{0}

It is common in machine learning to encounter optimization problems involving tens of millions of training examples and millions of variables.  To deal with the demands of time, storage and processing power imposed by such applications, high performance implementations of stochastic gradient and batch quasi-Newton methods have been developed; see e.g.,  \citep{tsitsiklis1986distributed,
recht2011hogwild,agarwal2014reliable,
dean2012large,zinkevich2010parallelized,
das2016distributed,berahas2019quasi}. In this paper we study a batch approach based on the L-BFGS method \citep{nocedal1980updating,LiuNocedal89} that strives to reach the right balance between efficient learning and productive parallelism.

At present, due to its fast learning properties and low per-iteration cost, the preferred method for very large scale applications is the stochastic gradient  (SG) method \citep{robbins1951stochastic,
bottou2003large}, and its variance-reduced and accelerated variants \citep{Schmidt2016,johnson2013accelerating,
defazio2014saga,nguyen2017sarahicml,
konevcny2016mini,
konevcny2017semi,
he2015dual,
nguyen2017stochastic,
lin2015universal}. These methods are implemented either in an asynchronous manner (e.g., using a parameter server in a distributed setting) or following a synchronous mini-batch approach that exploits parallelism in the gradient evaluations \citep{bertsekas1989parallel,recht2011hogwild,Goodfellow-et-al-2016,reddi2015variance,leblond2016asaga,takac2013mini}. A drawback of the asynchronous approach
is that it cannot use large batches, as this would cause updates to become too dense and compromise
the stability and scalability of the method \citep{recht2011hogwild,mania2015perturbed}. As a result, the algorithm spends more time in communication as compared to computation. On the other hand, using a synchronous mini-batch approach one can achieve a near-linear decrease in the number of SG iterations as the mini-batch size is increased, up to a certain point after which the increase in computation is not offset by the faster convergence \citep{takac2013mini}.

An alternative to SG-type methods are batch methods, such as L-BFGS \citep{nocedal1980updating}, because they parallelize well and are able to achieve high training accuracy. Batch methods allow for more computation per node, so as to achieve a better balance with the communication costs \citep{zhang2015disco,berahas2019quasi}; however, batch methods are not as efficient learning algorithms as SG methods in a sequential setting \citep{bousquet2008tradeoffs,hardt2016train}. To benefit from both types of methods, some high performance machine learning systems implement both types of methods  \citep{agarwal2014reliable,dean2012large}, and algorithms that transition from the stochastic to the batch regime \citep{friedlander2012hybrid,byrd2012sample,bollapragada2018adaptive,bollapragada2018progressive} have also received attention recently.

%The goal of this paper is to propose a single method that selects a \emph{sizeable} subset (batch) of the training data to compute a step, and changes this batch at every iteration to improve the learning abilities of the method.  We call this a \emph{multi-batch} approach to differentiate it from the mini-batch approach used in conjunction with the SG method, which employs a very small subset of the training data. In this regime, it is natural to employ a quasi-Newton method, as incorporating second-order information imposes  little computational overhead and improves the stability and speed of the method. The multi-batch approach can, however, cause difficulties to quasi-Newton methods because these methods employ gradient differences to update the Hessian approximations. 

The goal of this paper is to propose a single method that selects a \emph{sizeable} subset (batch) of the training data to compute a step and changes this batch at every iteration to improve the learning abilities of the method. In order to differentiate it from the mini-batch approach used in conjunction with the SG method, which employs a very small subset of the training data, we call this the \emph{multi-batch} approach. In this regime it is natural to employ a quasi-Newton method, as incorporating second-order information imposes  little computational overhead and improves the stability and speed of the method. However, the multi-batch approach can cause difficulties to quasi-Newton methods as these methods employ gradient differences to update the Hessian approximations.

%In this paper, we study how to design a  robust multi-batch implementation of the limited-memory version of the classical BFGS method \citep{broyden1967quasi,fletcher1970new,goldfarb1970family,shanno1970conditioning}---which we call the \emph{multi-batch L-BFGS} method---in the presence of two adverse situations \citep{Berahas2016}. The first occurs in parallel implementations when some of the computational nodes devoted to the evaluation of the function and gradient are unable to return results on time, i.e., in the presence of \emph{faults}.  This amounts to using different data points to evaluate the function and gradient at the beginning and the end of the iteration, which can be  harmful to quasi-Newton methods since they employ gradient differences  to update Hessian approximations.  A similar challenge occurs in a \emph{multi-batch} approach in which the data points used to compute the function and gradient are purposely changed at each iteration (or every several iterations) to  accelerate the learning process. The main objective of this paper is to show that stable quasi-Newton updating can be achieved in these settings without incurring extra computational cost or  special synchronization. The key is to perform quasi-Newton updating based on the \emph{overlap} between consecutive batches. The only restriction is that this overlap should not be insignificant, something that can be expected (or easily enforced) in most situations.

More specifically, in this paper we study how to design a \emph{robust multi-batch} implementation of the limited-memory version of the classical BFGS method \citep{broyden1967quasi,fletcher1970new,goldfarb1970family,shanno1970conditioning}---which we call the \emph{multi-batch L-BFGS} method---in the presence of two adverse situations \citep{Berahas2016,pei2019single,pei2019two,karakus2017straggler,karakus2018redundancy}. The first occurs in parallel implementations when some of the computational nodes devoted to the evaluation of the function and gradient are unable to return results on time, i.e., in the presence of \emph{faults}.  This amounts to using different data points to evaluate the function and gradient at the beginning and the end of the iteration, which can be  harmful to quasi-Newton methods since they employ gradient differences  to update Hessian approximations.  A similar challenge occurs in a \emph{multi-batch} approach in which the data points used to compute the function and gradient are purposely changed at each iteration (or every several iterations) to  accelerate the learning process. The main objective of this paper is to show that \emph{stable quasi-Newton updating} can be achieved in these settings without incurring extra computational cost or  special synchronization. The key is to perform quasi-Newton updating based on the \emph{overlap} between consecutive batches. The only restriction is that this overlap should not be insignificant, something that can be expected, or easily enforced, in most situations.

Recently, several stochastic quasi-Newton (SQN) methods have been proposed; see e.g., \citep{schraudolph2007stochastic,bordes2009sgd,Sammy_SQN,curtis2016self,Wang2016,mokhtari2015global,keskar2016adaqn,gower2016stochastic,Berahas2016}. The methods enumerated above differ in three major aspects: $(i)$ the update rules for the curvature (correction) pairs and the Hessian approximation, $(ii)$ the frequency of updating, and $(iii)$ the required extra computational cost and synchronization required. Our method is different from these methods predominantly due to the fact that it does not modify the BFGS update equations or the form of the curvature pairs, and does not require extra (gradient) computations.  Additionally, our method is designed to work in a distributed settings with faults, in which faults occur randomly and sample consistency cannot be assumed, and as such several SQN methods are not suitable.

%\textcolor{red}{In this paper}, 
We analyze the convergence properties of the multi-batch L-BFGS method using a fixed step length strategy, as well as a diminishing step length strategy, on both strongly convex and nonconvex problems. This is appropriate in our setting, as using a fixed step length approach is popular in practice, and facilitates the study of the stability of quasi-Newton updating in a distributed setting. For strongly convex functions, we show that the algorithm converges, at a linear rate, to an approximate solution whose accuracy depends on the variance of the gradients and the step length. In the nonconvex setting, we show that if cautious BFGS updating is employed, the expected value of the average norm-squared of the gradient is bounded. 

We present numerical experiments on a plethora of problems that arise in machine learning and deep learning. We first illustrate the robustness of our proposed approach on binary classification logistic regression problems on a distributed computing platform with faults and in the serial multi-batch setting. The results indicate that the proposed method achieves a good balance between computation and communication costs. Moreover, we present results on neural network training tasks that illustrate that when larger batch-size is used, our algorithm is competitive with the state-of-the-art. Finally, we demonstrate the strong and weak scaling properties of the proposed method.

\medskip 

The paper is organized as follows. In Section \ref{sec:multi} we describe the multi-batch L-BFGS method in detail. In Section \ref{sec:conv_anal} we provide convergence analyses for the proposed method for strongly convex and nonconvex functions. Numerical results that illustrate the practical performance and robustness of the multi-batch L-BFGS method are reported in Section \ref{sec:num_res}. Finally, in Section \ref{sec:final_rem} we provide some concluding remarks.

\section{A Multi-Batch Quasi-Newton Method}
\label{sec:multi}
\setcounter{equation}{0}

Ideally, in supervised learning, one seeks to minimize  expected risk, defined as
\begin{align} \label{eq:erisk}
    R(w)= \int_\Omega f(w; x, y) d P(x, y) = \mathbb{E}[f(w;x,y)],
 \end{align}
 where $(x,y)$ are input-output pairs, $f : \mathbb{R}^d \rightarrow \mathbb{R}$ is the composition of a prediction function (parametrized by $w$) and a loss function, and $\Omega$ is the space of input-output pairs endowed with a probability distribution $P(x,y)$. Since the distribution $P$ is typically not known, one  approximates \eqref{eq:erisk} by the empirical risk
 \begin{align*}  %\label{eq:empirical}
  F(w) = \frac{1}{n}\sum_{i=1}^{n}f(w;x^{i},y^{i}) \defeq \frac{1}{n}\sum_{i=1}^{n}f_i(w),
\end{align*}
where $ (x^i, y^i)$, for $i=1, \ldots, n$, denote the training examples, also referred to as data points or samples. The training problem consists of finding an optimal choice of the parameters $w \in \mathbb{R}^d$ with respect to $F$, i.e., to compute a solution of the problem
\begin{align}  \label{eq:obj}
   \min_{w\in\mathbb{R}^d}F(w) =\frac{1}{n}\sum_{i=1}^{n}f_i(w).
\end{align}

%In a pure batch approach, one applies a gradient-based method to this deterministic optimization problem \eqref{eq:obj}. A popular method in this setting is the L-BFGS method \citep{LiuNocedal89}. When $n$ is large, it is natural to parallelize the evaluation of $F$ and $\nabla F$ by assigning the evaluation of the  component functions $f_i$ to different processors. If this is done in a distributed platform, it is possible for some of the computational nodes dedicated to a portion of the evaluation of the objective function and the gradient to be  slower than the rest. In this case, the contribution of the slow (or unresponsive) computational nodes could  be ignored given the stochastic nature of the true objective function \eqref{eq:erisk}. This leads, however, to an inconsistency in the objective function and gradient at the beginning and at the end of the iteration, which can be detrimental to quasi-Newton methods, as mentioned above. Thus, we seek to find a \emph{fault-tolerant} version of the batch L-BFGS method that is capable of dealing with slow or unresponsive computational nodes.

In a pure batch approach, one applies a gradient-based method to the deterministic optimization problem \eqref{eq:obj}. In this regime, a popular method is L-BFGS \citep{nocedal1980updating,LiuNocedal89}. When $n$ is large, it is natural to parallelize the computation of $F$ and $\nabla F$ by assigning the evaluation of component functions $f_i$, or subsets of the component functions, to different processors. If this is done on a distributed computing platform, it is possible for some of the computational nodes, dedicated to a portion of the evaluation of the objective function and the gradient, to be slower than the rest. In this case, the contribution of the slow (or unresponsive) computational nodes could potentially be ignored given the stochastic nature of the true objective function \eqref{eq:erisk}. However, this leads to an inconsistency in the objective function and gradient at the beginning and at the end of the iteration, which can be detrimental to quasi-Newton methods, as mentioned above. Hence, we seek to develop a \emph{fault-tolerant} version of the batch L-BFGS method that is capable of dealing with slow or unresponsive computational nodes.

%A similar challenge arises in a \emph{multi-batch} implementation of the L-BFGS method in which the entire training set, $ T= \{ (x^i, y^i)$, for $i=1, \ldots, n\}$, is not employed at every iteration but, rather,  a subset of the data is used to compute the gradient. Specifically, we consider a method in which the dataset is randomly divided into  a number of batches---say 10, 50, or 100---and  the minimization is performed with respect to a different batch at every iteration.  At the $k$-th iteration the algorithm chooses $S_k \subset \{1, \ldots, n\}$, computes 
%\begin{align}   \label{eq:batch_fg}
	%F^{S_{k}}(w_k)=\frac{1}{\left|S_{k}\right|}\sum_{i\in S_{k}}f_i\left(w_{k}\right), \qquad {g}_{k}^{S_{k}}=\nabla F^{S_{k}}(w_k) = \frac{1}{\left|S_{k}\right|}\sum_{i\in S_{k}}\nabla f_i\left(w_{k}\right) ,
%\end{align}
%and takes a step along the direction $- H_k g_k^{S_k}$, where $H_k$ is an approximation to $\nabla^2 F(w_k)^{-1}$. Allowing the sample $S_k$ to change freely at every iteration gives this approach  flexibility of implementation and is beneficial to the learning process. Note, we refer to $S_k$ as the sample of training points, even though $S_k$ only indexes those points.

A similar challenge arises in a \emph{multi-batch} implementation of the L-BFGS method in which only a subset of the data is used to compute the gradient at every iteration. We consider a method in which the dataset is randomly divided into  a number of batches and the minimization is performed with respect to a different batch at every iteration. Specifically, at the $k$-th iteration the algorithm chooses $S_k \subset \{1, \ldots, n\}$, computes 
\begin{align}   \label{eq:batch_fg}
F^{S_{k}}(w_k)=\frac{1}{\left|S_{k}\right|}\sum_{i\in S_{k}}f_i\left(w_{k}\right), \qquad {g}_{k}^{S_{k}}=\nabla F^{S_{k}}(w_k) = \frac{1}{\left|S_{k}\right|}\sum_{i\in S_{k}}\nabla f_i\left(w_{k}\right) ,
\end{align}
and takes a step along the direction $- H_k g_k^{S_k}$, where $H_k$ is an approximation to $\nabla^2 F(w_k)^{-1}$. Allowing the sample $S_k$ to change freely at every iteration gives this approach flexibility and is beneficial to the learning process. Note, we refer to $S_k$ as the sample of training points, even though $S_k$ only indexes those points.

%The case of unresponsive computational nodes and the  {multi-batch} method are similar. The main difference is that node failures create unpredictable changes to the samples $S_k$, whereas a multi-batch method has control over sample generation. In either case, the algorithm employs a stochastic approximation to the gradient and can no longer be considered deterministic.  We must, however, distinguish our setting from that of the classical SG method, which employs small mini-batches  and noisy gradient approximations. Our algorithm operates with much larger batches so that distributing the function and gradient evaluations is beneficial and the compute time of $g_k^{S_k}$ is not overwhelmed by communication costs. This gives rise to gradients with relatively small variance and justifies the use of a second-order method such as L-BFGS.

The case of unresponsive computational nodes and the multi-batch regime are similar in nature, i.e., the samples $S_k$ used change from one iteration to the next. The main difference is that node failures create unpredictable changes to the samples, whereas a multi-batch method has control over the sample generation. In either case, the algorithm employs a stochastic approximation to the gradient and can no longer be considered deterministic. We must, however, distinguish our setting from that of the classical SG method, which employs small mini-batches. Our algorithm operates with much larger batches so that distributing the computation of the function and gradient is beneficial, and the compute time is not overwhelmed by communication costs. This gives rise to gradients with relatively small variance and justifies the use of a second-order method such as L-BFGS.

The robust implementation of the L-BFGS method, proposed in \cite{Berahas2016}, is based on the following observation:  The  difficulties created by the use of a different sample $S_k$ at each iteration can be circumvented if consecutive samples $S_{k}$ and $S_{k+1}$ have an overlap, so that
\begin{align*}
      O_k= S_{k} \cap S_{k+1} \neq \emptyset.
\end{align*}
 One can then perform \emph{stable quasi-Newton updating} by computing gradient differences based on this overlap, i.e.,  by defining 
\begin{align}   \label{pairs}
       y_{k+1}=g_{k+1}^{O_{k}}-g_{k}^{O_{k}}, \qquad s_{k+1} = w_{k+1}-w_k,
\end{align}
in the notation given in \eqref{eq:batch_fg}, and using this correction pair $(y_k, s_k)$ in the BFGS update. When the overlap set $O_k$ is not too small, $y_k$ is a useful approximation of the curvature of the objective
function along the most recent displacement, and leads to a productive quasi-Newton step. This observation is based on an important property of Newton-like methods, namely that there is much more freedom in choosing a Hessian approximation than in computing the gradient \citep{byrd2011use,bollapragada2016exact,martens2010deep,berahas2017investigation}. More specifically, a smaller sample $O_k$ can be employed for updating the inverse Hessian approximation $H_k$, than for computing the batch gradient $g_k^{S_k}$ used to define the search direction $- H_k g_k^{S_k}$. In summary, by ensuring that unresponsive nodes do not constitute the vast majority of all compute nodes in a {fault-tolerant} parallel implementation, or by exerting a small degree of control in the creation of the samples $S_k$ in the multi-batch regime, one can design a robust method that naturally builds upon the fundamental properties of BFGS updating.

We should mention that a  commonly used fix  for ensuring stability of quasi-Newton updating in machine learning is to enforce gradient consistency \citep{schraudolph2007stochastic,mokhtari2015global}, i.e., to use the same sample $S_k$ to compute gradient evaluations at the beginning and the end of the iteration, at the cost of double gradient evaluations.  Another popular remedy is to use the same batch $S_k$ for multiple iterations \citep{ngiam2011optimization}, alleviating the gradient inconsistency problem at the price of slower convergence. In this paper, we assume that such \emph{sample consistency is not possible} (the fault-tolerant case) or \emph{desirable} (the multi-batch regime), and wish to design and analyze an implementation of L-BFGS that imposes minimal restrictions in the changes of the sample.

\subsection{Specification of the Method}
Let us begin by considering a robust implementation of the multi-batch BFGS method and  then consider its limited memory version. At  the $k$-th iteration, the multi-batch BFGS algorithm chooses a set $S_k \subset \{1, \ldots, n\}$ and computes a new iterate by the formula
\begin{align}  \label{eq:update}
w_{k+1}=w_k-\alpha_{k}H_{k} g_{k}^{S_{k}} ,
\end{align}
where $\alpha_{k}$ is the step length, $g_k^{S_k}$ is the batch gradient \eqref{eq:batch_fg} and $H_{k}$ is the inverse BFGS
Hessian matrix approximation that is updated at every iteration by means
of the formula
\begin{equation*}   %\label{eq:bfgs}
H_{k+1}=V_{k}^{T}H_{k}V_{k}+\rho_{k}s_{k}s_{k}^{T}, \qquad \rho_{k}=\frac{1}{y_{k}^{T}s_{k}}, \qquad V_{k}=1-\rho_{k}y_{k}s_{k}^{T} .
\end{equation*}
 To compute the correction vectors $(s_k, y_k)$, we determine the overlap set $O_k = S_k \cap S_{k+1}$ consisting of the samples that are common at the $k$-th and $k+1$-st iterations. We define  
 \begin{equation*}  % \label{eq:over_fg}
         %{F}^{O_{k}}(w_k)=\frac{1}{\left|O_{k}\right|}\sum_{i\in O_{k}}f_i\left(w_k\right), \qquad 
         {g}_{k}^{O_{k}}=\nabla F^{O_k}(w_k)=\frac{1}{\left|O_{k}\right|}\sum_{i\in O_{k}}\nabla f_i\left(w_k\right),
\end{equation*}
and compute the correction pairs as in \eqref{pairs}. This completely specifies the algorithm, except for the choice of step length $\alpha_k$; in this paper we consider constant and diminishing step lengths.

In the limited memory version, the matrix $H_k$ is defined at each iteration as the result of applying $m$ BFGS updates  to a multiple of the identity matrix, using a set of $m$ correction pairs $\{s_i, y_i\}$ kept in storage. The memory parameter $m$ is typically in the range 2 to 20.  When computing the search direction (matrix-vector product) in \eqref{eq:update} it is not necessary to form the dense matrix $H_k$ since one can obtain this product via the two-loop recursion \citep{mybook}, using the $m$ most recent correction pairs. Employing this mechanism, the search direction can be computed in $\mathcal{O}(d)$ floating operations, where $d$ is the number of variables. After the step has been computed, the oldest pair $(s_j, y_j)$ is discarded and the new curvature pair is stored.

A pseudo-code of the multi-batch limited-memory BFGS algorithm is given in Algorithm \ref{alg:multi}, and depends on several parameters.  The parameter $r$ denotes the fraction of samples in the dataset used to define the gradient, i.e., $r = \frac{\left| S\right|}{n}$. The parameter $o$ denotes the length of overlap between consecutive samples, and is defined  as a fraction of the number of samples in a given batch $S$, i.e., $o = \frac{\left| O\right|}{\left| S\right|}$.

\begin{algorithm}[]
\caption{Multi-Batch L-BFGS}
  \label{alg:multi}
 {\bf Input:} $w_{0}$ (initial iterate), $m$ (memory parameter), $r$ (batch, fraction of $n$), $o$ (overlap, fraction of batch), $k\leftarrow0$ (iteration counter). %Define $S = \{1, \ldots, n\}$.

  \begin{algorithmic}[1]
  %\State Shuffle dataset $S$
  \State Create initial batch $S_{0}$ \label{step:samp1}%\Comment{As shown in Firgure~\ref{fig:sample_creation}}
  %\State Evaluate $F_{0}^{S_{0}}$ and $g_{0}^{S_{0}}$ \Comment{As defined in \eqref{eq:batch_fg}}
  \For {$k=0,1,2,...$}
\State Calculate the search direction $p_{k}=-H_{k}g_{k}^{S_{k}}$ %\Comment{Using L-BFGS formula}
\State Choose a step length $\alpha_{k} >0$
\State Compute  $w_{k+1}=w_k+\alpha_{k}p_{k}$ 
\State Create the next batch $S_{k+1}$ \label{step:samp2}
%\State Evaluate $g_{k+1}^{O_{k}}$ and $g_{k}^{O_{k}}$
\State Compute the curvature pairs $s_{k+1}=w_{k+1}-w_k$ and
$y_{k+1}=g_{k+1}^{O_{k}}-g_{k}^{O_{k}}$ 
\State Replace the oldest  pair $(s_i, y_i)$  by $s_{k+1}, y_{k+1}$ (if $m$ pairs stored, else just add)
%\State Compute the new Hessian approximation $H_{k+1}$ \Comment{As defined in \eqref{eq:bfgs}}
%\State Evaluate $F_{k+1}^{S_{k+1}}$ and $g_{k+1}^{S_{k+1}}$ 
%\State Increment iteration counter $k\leftarrow k+1$
\EndFor
  \end{algorithmic}
\end{algorithm}

\subsection{Sample Generation}
\label{sec:sampling}

The \textit{fault-tolerant} and \textit{multi-batch} settings differ in the way the samples $S_k$ and $O_k$ are formed (Lines \ref{step:samp1} \& \ref{step:samp2}, Algorithm \ref{alg:multi}). In the former, sampling is done automatically as a by-product of the nodes that fail to return a computation (gradient evaluation). In the latter, the samples $S_k$ and $O_k$ used at every iteration are purposefully changed in order to accelerate the learning process, thus sampling is user controlled. In either setting, independent sampling can be achieved; a necessary condition to establish convergence results. We first describe the \textit{fault-tolerant} setting, and then propose two sampling strategies that can be employed in the \textit{multi-batch} setting. Let $ T= \{ (x^i, y^i)$, for $i=1, \ldots, n\}$ denote the training set.

\begin{figure}[]
\begin{center}
\subfigure[Fault-tolerant sampling\label{ft_samp}]{
\resizebox*{6.5cm}{!}{\includegraphics{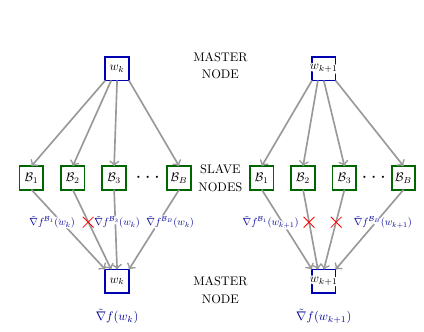}}}\hspace{2.5pt}
\subfigure[Multi-batch sampling\label{mb_samp}]{
\resizebox*{6.5cm}{!}{\includegraphics{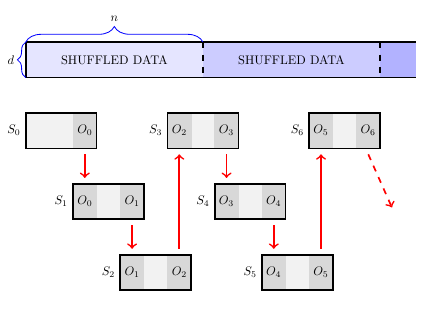}}}
\caption{Sample and Overlap formation for two adversarial situations.} \label{fig:sample_creation}
\end{center}
\end{figure}

\paragraph*{\textbf{Fault-Tolerant}}
Consider a distributed implementation in which slave nodes read the current iterate $w_k$ from the master node, compute a local gradient on a subset of the dataset, and send it back to the master node for aggregation in the calculation \eqref{eq:batch_fg}.  Given a time (computational) budget, it is possible for some nodes to fail to return a result. The schematic in Figure~\ref{ft_samp} illustrates the gradient calculation across two iterations, $k$ and $k+1$, in the presence of faults.  Here,  $B$ is the total number of slave nodes, $\mathcal{B}_i$ for $i=1,...,B$ denote the batches of data that each slave node $i$ receives ($T = \cup_i \mathcal{B}_i$),  and 
$\tilde{\nabla}f(w)$ is the gradient calculation using all nodes that responded within the preallocated time.

Let $\mathcal{J}_k\subset \{1,2,...,B\}$ and $\mathcal{J}_{k+1}\subset \{1,2,...,B\}$ be the set of indices of all nodes that returned a gradient at the $k$-th and $k+1$-st iterations, respectively. Using this notation $S_k = \cup_{j\in \mathcal{J}_k} \mathcal {B}_j$ and $S_{k+1} = \cup_{j\in \mathcal{J}_{k+1}} \mathcal {B}_j$, and we define $O_k = \cup_{j \in \mathcal{J}_k\cap \mathcal{J}_{k+1}} \mathcal {B}_j$. The simplest implementation in this setting preallocates the data on each compute node, requiring minimal data communication, i.e., only one data transfer. In this case, the samples $S_k$ are independent if node failures occur randomly. On the other hand, if the same set of nodes fail, then the sample creation will be biased, which is harmful both in theory and in practice.  One way to ensure independent sampling is to shuffle and redistribute the data to all nodes after every iteration or after a certain number of iterations.

\paragraph*{\textbf{Multi-Batch Sampling}}
In the \textit{multi-batch} setting several strategies can be employed, with the only restriction that consecutive batches $S_k$ and $S_{k+1}$ should, to a certain degree, overlap. We propose two sampling strategies: $(i)$ overlaps $O_k$ are forced in the sample creation process, $(ii)$ the overlapping set $O_k$ is subsampled from the batch $S_k$. In practice the two strategies perform on par, however, there is a subtle difference. In the second strategy the batches are sampled independently, something that is not true for the strategy in which overlapping samples are forced. The independent sampling strategy of course does not come for free as this strategy incurs an increase in computational cost per iteration. However, as mentioned above, the overlapping set $O_k$ need not be very large, and thus the increase in cost is negligible as compared to the rest of the computation. We now describe the two approaches in more detail.

Figure~\ref{mb_samp} illustrates the sample creation process in the first strategy. The dataset is shuffled and batches are generated by collecting subsets of the training set, in order. Every set (except $S_0$) is of the form $S_k= \{ O_{k-1}, N_k, O_k\}$, where $O_{k-1}$ and $O_k$ are the overlapping samples with batches $S_{k-1}$ and $S_{k+1}$ respectively, and $N_k$ are the samples that are unique to batch $S_k$. After each pass through the dataset, the samples are reshuffled, and the procedure described above is repeated. In our implementation samples are drawn without replacement, guaranteeing that after every epoch (pass over the whole dataset) all samples are used. This strategy has the advantage that it requires no extra computation in the evaluation of $g_k^{O_k}$ and $g_{k+1}^{O_k}$, but the samples $S_k$ are not independent.

The second sampling  strategy is simpler and requires less control. At every iteration $k$, a batch $S_k$ is created by randomly selecting $\left| S_k \right|$ elements from $\{1,\ldots n\}$. The set $O_k$ is then formed by randomly selecting $\left| O_k \right|$ elements from $S_k$ (subsampling). Note, in this sampling strategy the samples $O_k$ need not be in the set $ S_{k+1} $. This strategy is slightly more expensive since $g_{k+1}^{O_k}$ requires extra computation, but if the overlap is small this cost is not significant. 

%%% Convergence Analysis %%%%
\section{Convergence Analysis}
\label{sec:conv_anal}
\setcounter{equation}{0}

In this Section, we analyze the convergence properties of the multi-batch L-BFGS method (Algorithm \ref{alg:multi}) when applied to the minimization of \emph{strongly convex}  and \emph{nonconvex} objective functions, using a fixed step length strategy, as well as a diminishing step length strategy. We assume that the goal is to minimize the empirical risk $F$ \eqref{eq:obj}, but note that a similar analysis could be used to study the minimization of the expected risk \eqref{eq:erisk}.  

\subsection{Strongly Convex case}
Due to the stochastic nature of the multi-batch approach, every iteration of Algorithm~\ref{alg:multi} employs a gradient that contains errors that do not converge to zero. Therefore, by using a fixed step length strategy one cannot establish convergence to the optimal solution $w^{\star}$, but only convergence to a neighborhood of $w^{\star}$ \citep{nedic2001convergence}. Nevertheless, this result is of interest as it reflects the common practice of using a fixed step length and decreasing it only if the desired testing error has not been achieved. It also illustrates the tradeoffs that arise between the size of the batch and the step length.

In our analysis, we make the following assumptions about the objective function and the algorithm.

%% Assumptions A %%
\paragraph*{\textbf{Assumptions A}}
\emph{ 
\begin{enumerate} 
\item $F$ is twice continuously differentiable.
\item There exist positive constants $\hat{\lambda}$ and $\hat{\Lambda}$ such that
$\hat{\lambda} I \preceq \nabla^2F^O(w) \preceq \hat{\Lambda} I$
for all $w \in \mathbb{R}^d$ and all sets $O \subset  \{1,2,\ldots,n\}$ of length $|O| = o\cdot r\cdot n$.
\item There exist constants $\gamma \geq 0$ and $\eta \geq 1$ such that $\mathbb{E}_{S}\left[ \| \nabla  F^{S}(w) \|^2 \right] \leq \gamma^2 + \eta \| \nabla F(w)\|^2$
for all $w \in \mathbb{R}^d$ and all sets $S\subset  \{1,2,\ldots,n\}$  of length $|S|=r \cdot n$.
\item The samples $S$ are drawn independently and $\nabla F^{S}(w)$ is an unbiased estimator of the true gradient $\nabla F(w)$ for all $w \in \mathbb{R}^d$, i.e.,
$
\mathbb{E}_{S}[ \nabla F^{S}(w)] = \nabla F(w).$
\end{enumerate}}
\noindent
Note that Assumption $A.2$ implies that the entire Hessian $\nabla^2F(w)$ also satisfies 
\begin{align}  \label{eq:b_hess1}
 \lambda I \preceq \nabla^2F(w) \preceq  \Lambda I,  \quad \forall w \in \mathbb{R}^d,
\end{align}
for some constants $ \lambda,  \Lambda>0$. Assuming that every subsampled function $F^O(w)$ is strongly convex is not unreasonable  as a regularization term is commonly added in practice when that is not the case.

\medskip

We begin by showing that the inverse Hessian approximations $H_k$ generated by the multi-batch L-BFGS method have eigenvalues that are uniformly bounded above and away from zero. The proof technique used is an adaptation of that in \citep{Sammy_SQN}.

\begin{lemma}	\label{lemma1}
If Assumptions A.1 \& A.2 hold, there exist constants $0<\mu_1\leq \mu_2$ such that the inverse Hessian approximations $\{H_k\}$ generated by Algorithm~\ref{alg:multi} satisfy
\begin{align*}    %\label{eq:lemma}
\mu_1 I \preceq H_k \preceq \mu_2 I,\qquad \text{for } k=0,1,2,\dots 
\end{align*}
\end{lemma}

\begin{proof}
Instead of analyzing the inverse Hessian approximation $H_k$, we study the Hessian approximation $B_k = H_k^{-1}$. In this case, the limited memory quasi-Newton updating formula is given as follows:
\begin{enumerate}
\item Set $B_k^{(0)}=\frac{y_k^Ty_k}{s_k^Ty_k}I$ and $\tilde{m} = \min\{k,m\}$; where $m$ is the memory in L-BFGS.
\item For $i=0,...,\tilde{m}-1$ set $j=k-\tilde{m}+1+i$ and compute
\begin{align*}    %\label{eq:bfgs_upd}
B_k^{(i+1)}=B_k^{(i)}-\frac{B_k^{(i)}s_js_j^TB_k^{(i)}}{s_j^TB_k^{(i)}s_j} + \frac{y_jy_j^T}{y_j^Ts_j}.
\end{align*}
\item Set $B_{k+1} = B_k^{(\tilde{m})}.$
\end{enumerate}

The curvature pairs $s_k$ and $y_k$ are updated via the following formulae
\begin{align}    \label{eq:curv_upd}
y_{k+1}=g_{k+1}^{O_{k}}-g_{k}^{O_{k}}, \qquad s_{k+1} = w_{k+1}-w_k.
\end{align}
A consequence of Assumption $A.2$ is that the eigenvalues of any sub-sampled Hessian ($\left| O \right|$ samples) are bounded above and away from zero. Utilizing this fact, the convexity of component functions and the definitions  \eqref{eq:curv_upd}, we have
\begin{align} \label{eq:expr1}
y_k^Ts_k \geq \frac{1}{\hat{\Lambda}}\|y_k\|^2  \quad & \Rightarrow \quad \frac{\|y_k\|^2 }{y_k^Ts_k} \leq \hat{\Lambda}.
\end{align}
On the other hand, strong convexity of the sub-sampled functions, the consequence of Assumption $A.2$ and definitions  \eqref{eq:curv_upd}, provide a lower bound,
\begin{align} \label{eq:expr2}
y_k^Ts_k \leq \frac{1}{\hat{\lambda}}\|y_k\|^2  \quad & \Rightarrow \quad \frac{\|y_k\|^2 }{y_k^Ts_k} \geq \hat{\lambda}.
\end{align}   
%By convexity and strong convexity of $F$ and $f_i$ for all $i=1,...,N$, and utilizing a consequence of assumption \emph{A2} and definitions \eqref{eq:curv_upd}, we have the following two expressions, 
%\begin{align*} %\label{eq:expr}
%\begin{split}
%y_k^Ts_k \geq \frac{1}{\tilde{\Lambda}}\|y_k\|^2  \quad & \Rightarrow \quad \frac{\|y_k\|^2 }{y_k^Ts_k} \leq \tilde{\Lambda}, \nonumber\\
%y_k^Ts_k \leq \frac{1}{\tilde{\lambda}}\|y_k\|^2  \quad & \Rightarrow \quad \frac{\|y_k\|^2 }{y_k^Ts_k} \geq \tilde{\lambda}
%\end{split}
%\end{align*}   
Combining the upper and lower bounds \eqref{eq:expr1} and \eqref{eq:expr2}  
\begin{align}    \label{eq:bound}
\hat{\lambda} \leq \frac{\|y_k\|^2}{y_k^Ts_k} \leq \hat{\Lambda}.
\end{align}

The above proves that the eigenvalues of the matrices $B_k^{(0)}=\frac{y_k^Ty_k}{s_k^Ty_k}I$ at the start of the L-BFGS update cycles are bounded above and away from zero, for all $k$. We now use a Trace-Determinant argument to show that the eigenvalues of $B_k$ are bounded above and away from zero. Let $tr(B)$ and $\det(B)$ denote the trace and determinant of matrix $B$, respectively, and set $j_i = k-\tilde{m}+i$. The trace of the matrix $B_{k+1}$ can be expressed as, 
\begin{align} \label{eq:trace}
tr(B_{k+1}) &= tr\left(B_k^{(0)}\right) - tr\sum_{i=1}^{\tilde{m}}\left(\frac{B_k^{(i-1)}s_{j_i}s_{j_i}^TB_k^{(i-1)}}{s_{j_i}^TB_k^{(i-1)}s_{j_i}}\right) + tr\sum_{i=1}^{\tilde{m}} \frac{y_{j_i}y_{j_i}^T}{y_{j_i}^Ts_{j_i}}\nonumber\\
&\leq tr\left(B_k^{(0)}\right) + \sum_{i=1}^{\tilde{m}} \frac{\|y_{j_i}\|^2}{y_{j_i}^Ts_{j_i}}
%\nonumber\\ &
\leq tr\left(B_k^{(0)}\right) + \tilde{m}\hat{\Lambda} 
%\nonumber\\ &
 \leq C_1, 
\end{align}
for some positive constant $C_1$, where the inequalities above are due to \eqref{eq:bound}, and the fact that the eigenvalues of the initial L-BFGS matrix $B_k^{(0)}$ are bounded above and away from zero.

Using a result due to \cite{powell1976some}, the determinant of the matrix  $B_{k+1}$ generated by the multi-batch L-BFGS method can be expressed as, 
\begin{align}	\label{eq:det}
\det (B_{k+1}) &= \det \left(B_{k}^{(0)}\right) \prod_{i=1}^{\tilde{m}} \frac{y_{j_i}^Ts_{j_i}}{s_{j_i}^TB_{k}^{(i-1)}s_{j_i}} %\nonumber\\ & 
=  \det \left(B_{k}^{(0)}\right) \prod_{i=1}^{\tilde{m}} \frac{y_{j_i}^Ts_{j_i}}{s_{j_i}^Ts_{j_i}}  \frac{s_{j_i}^Ts_{j_i}}{s_{j_i}^TB_{k}^{(i-1)}s_{j_i}}\nonumber\\
& \geq  \det \left(B_{k}^{(0)}\right) \left( \frac{\hat{\lambda}}{C_1} \right)^{\tilde{m}} %\nonumber\\ &
 \geq C_2,
\end{align}
for some positive constant $C_2$, where the above inequalities are due to the fact that the largest eigenvalue of $B_{k}^{(i)}$ is less than $C_1$ and Assumption $A.2$.

The trace \eqref{eq:trace} and determinant \eqref{eq:det} inequalities derived above imply that the largest eigenvalues of all matrices $B_k$ are bounded from above, uniformly, and that the smallest eigenvalues of all matrices $B_k$ are bounded away from zero, uniformly. 
\end{proof}

Before we present the main theorem for the multi-batch L-BFGS method that employs constant step lengths, we state one more intermediate Lemma that bounds the distance between the function value at any point $w\in \mathbb{R}^d$ and the optimal function value with respect to the norm of the gradient squared.

\begin{lemma} 	\label{lemma1a}
Let Assumptions A.1 \& A.2 hold, and let $F^\star = F(w^\star)$, where $w^\star$ is the minimizer of $F$. Then, for all $w\in \mathbb{R}^d$,
\begin{align*}
	2\lambda(F(w)-F^\star) \leq \| \nabla F(w)\|^2.
\end{align*}
\end{lemma}

\begin{proof} As a result of Assumptions $A.1$, $A.2$ and \eqref{eq:b_hess1}, for all $x,y \in \mathbb{R}^d$
\begin{align*}
	F(x) \leq F(y) + \nabla F(y)^T (x-y) + \frac{1}{2\lambda}\| \nabla F(y) - \nabla F(x)\|^2;
\end{align*}
see \citep[Chapter~2.1.3]{nesterov2013introductory}.
Let $x=w$ and $y=w^\star$
\begin{align*}
F(w) &\leq F^\star + \nabla F(w^\star)^T (w-w^\star) + \frac{1}{2\lambda}\| \nabla F(w) - \nabla F(w^\star)\|^2 \\ 
& \leq F^\star + \frac{1}{2\lambda}\| \nabla F(w) \|^2.
\end{align*}
Re-arranging the above expression yields the desired result.
\end{proof}

Utilizing Lemmas \ref{lemma1} and \ref{lemma1a}, we show that the multi-batch L-BFGS method with a constant step length converges linearly to a neighborhood of the optimal solution.

\begin{theorem}
\label{thm:const}
Suppose that Assumptions A.1-A.4 hold, and let $F^{\star} = F(w^{\star})$, where $w^{\star}$ is the minimizer of $F$. Let $\{w_k\}$ be the iterates generated by Algorithm~\ref{alg:multi}, where $\alpha_k = \alpha$ satisfies
\begin{align}	\label{eq:steplength_sc}
0 <  \alpha  \leq \frac{\mu_1}{\mu_2^2 \eta \lambda},
\end{align}
and $w_0$ is the starting point.
Then for all $k\geq 0$,
\begin{align*}   
\mathbb{E}\big[ F(w_k) - F^{\star}\big] & \leq \big( 1-\alpha \mu_1 \lambda \big)^k \left[ F(w_0) - F^{\star} \right] + \left[ 1-(1-\alpha\mu_1 \lambda)^k\right]\frac{\alpha \mu_2^2 \gamma ^2 \Lambda}{2 \mu_1 \lambda}
\\    &
      \xrightarrow[]{k\rightarrow \infty} \frac{\alpha \mu_2^2 \gamma ^2 \Lambda}{2 \mu_1 \lambda}. 
\end{align*}
\end{theorem}

\begin{proof}
We have that
\begin{align} \label{eq:proof}
F(w_{k+1}) & = F(w_k -\alpha H_k \nabla F^{S_k}(w_k)) \nonumber \\
 & \leq F(w_k) + \nabla F(w_k)^T (-\alpha H_k \nabla F^{S_k}(w_k)) + \frac{\Lambda}{2}\| \alpha H_k \nabla F^{S_k}(w_k)\|^2 \nonumber \\
 & \leq F(w_k) - \alpha \nabla F(w_k)^T  H_k \nabla F^{S_k}(w_k) + \frac{\alpha^2 \mu_2^2 \Lambda}{2} \| \nabla F^{S_k}(w_k)\|^2,
\end{align}
where the first inequality arises due to \eqref{eq:b_hess1}, and the second inequality arises as a consequence of Lemma \ref{lemma1}. Taking the expectation (over $S_k$) of equation \eqref{eq:proof}
\begin{align}
\mathbb{E}_{S_k}[ F(w_{k+1})] & \leq F(w_k) - \alpha \nabla F(w_k)^T  H_k \nabla F(w_k) + \frac{\alpha^2 \mu_2^2 \Lambda}{2} \mathbb{E}_{S_k} \left[ \| \nabla F^{S_k}(w_k)\|^2 \right] \nonumber\\
& \leq F(w_k) - \alpha \mu_1 \| \nabla F(w_k) \|^2  + \frac{\alpha^2 \mu_2^2 \Lambda}{2}\left( \gamma^2 + \eta \| \nabla F(w)\|^2\right) \nonumber\\
& = F(w_k) - \alpha \left( \mu_1 - \frac{\alpha \mu_2^2 \eta \Lambda}{2}\right)\| \nabla F(w_k) \|^2  + \frac{\alpha^2 \mu_2^2 \gamma^2\Lambda}{2} \label{eq:proof3}\\
& \leq F(w_k) - \frac{\alpha \mu_1}{2} \| \nabla F(w_k) \|^2  + \frac{\alpha^2 \mu_2^2 \gamma^2\Lambda}{2} \label{eq:proof_2}
\end{align}
where the first inequality makes use of Assumption $A.4$, the second inequality arises due to Lemma \ref{lemma1} and Assumption $A.3$ and the third inequality is due to the step length \eqref{eq:steplength_sc}. Since $F$ is $\lambda$-strongly convex, we can substitute the result of Lemma \ref{lemma1a} in \eqref{eq:proof_2},
\begin{align} \label{eq:proof_3}
%\begin{split}
\mathbb{E}_{S_k} [F(w_{k+1})] & \leq F(w_k) - \frac{\alpha \mu_1}{2} \| \nabla F(w_k) \|^2  + \frac{\alpha^2 \mu_2^2 \gamma^2 \Lambda}{2} \nonumber \\
& \leq F(w_k) - \alpha \mu_1 \lambda  [F(w_k) - F^{\star}]  + \frac{\alpha^2 \mu_2^2 \gamma^2 \Lambda}{2}.
%\end{split}
\end{align}
Let 
\begin{align}			\label{eq:phi}
\phi_k = \mathbb{E}[F(w_k) - F^{\star}],
\end{align}
where the expectation is over all batches $S_0,S_1,...,S_{k-1}$ and all history starting with $w_0$. Equation \eqref{eq:proof_3} can be expressed as,
\begin{align}    \label{eq:proof_4}
\phi_{k+1} \leq (1 - \alpha \mu_1 \lambda ) \phi_k + \frac{\alpha^2 \mu_2^2 \gamma^2 \Lambda}{2}.
\end{align}
Since the step length is chosen according to \eqref{eq:steplength_sc} we deduce that $0 \leq (1 - \alpha \mu_1 \lambda ) < 1$.

Subtracting $\frac{\alpha\mu_2^2\gamma^2\Lambda}{2\mu_1 \lambda}$ from either side of \eqref{eq:proof_4} yields
\begin{align}    \label{eq:proof4.1}
%\begin{split}
\phi_{k+1} - \frac{\alpha\mu_2^2\gamma^2\Lambda}{2\mu_1 \lambda} & \leq (1 - \alpha \mu_1 \lambda ) \phi_k + \frac{\alpha^2 \mu_2^2 \gamma^2 \Lambda}{2} - \frac{\alpha\mu_2^2\gamma^2\Lambda}{2\mu_1 \lambda} \nonumber \\
& = (1 - \alpha \mu_1 \lambda ) \left[ \phi_k  - \frac{\alpha\mu_2^2\gamma^2\Lambda}{2\mu_1 \lambda} \right].
%\end{split}
\end{align}
Recursive application of \eqref{eq:proof4.1} yields
\begin{align*}   % \label{eq:proof_5}
%\begin{split}
\phi_{k} - \frac{\alpha\mu_2^2\gamma^2\Lambda}{2\mu_1 \lambda}& \leq  (1 - \alpha \mu_1 \lambda )^k \left[ \phi_0  - \frac{\alpha\mu_2^2\gamma^2\Lambda}{2\mu_1 \lambda} \right],
%\end{split}
\end{align*}
and thus, 
\begin{align*}	%\label{eq: blah}
	\phi_{k} \leq (1 - \alpha \mu_1 \lambda )^k \phi_0  + \left[ 1-(1-\alpha\mu_1 \lambda)^k\right]\frac{\alpha \mu_2^2 \gamma ^2 \Lambda}{2 \mu_1 \lambda}.
\end{align*}
Finally using the definition of $\phi_k$ \eqref{eq:phi} with the above expression yields the desired result
\begin{align*}  
\mathbb{E}\left[ F(w_k) - F^{\star}\right] \leq \left( 1-\alpha \mu_1 \lambda \right)^k \left[ F(w_0) - F^{\star} \right] + \left[ 1-(1-\alpha\mu_1 \lambda)^k\right]\frac{\alpha \mu_2^2 \gamma ^2 \Lambda}{2 \mu_1 \lambda}. 
\end{align*}
\end{proof}

The bound provided by this theorem  has two components: $(i)$ a term decaying linearly to zero, and $(ii)$ a  term identifying the neighborhood of convergence. Note, a larger step length yields a more favorable constant in the linearly decaying term, at the cost of an increase in the size of the neighborhood of convergence.  We consider these tradeoffs further in Section~\ref{sec:num_res}, where we also note that larger batch sizes increase the opportunities for parallelism and improve the limiting accuracy in the solution, but slow down the learning abilities of the algorithm. We should also mention that unlike the first-order variant of the algorithm ($H_k = I$), the step length range prescribed by the multi-batch L-BFGS method depends on $\mu_1$ and $\mu_2$, the smallest and largest eigenvalues of the L-BFGS Hessian approximation. In the worst-case, the presence of the matrix $H_k$ can make the limit in Theorem \ref{thm:const} significantly worse than that of the first-order variant if the update has been unfortunate and generates ill-conditioned matrices. We should note, however, such worst-case behavior is almost never observed in practice for BFGS updating.

 One can establish convergence of the multi-batch L-BFGS method to the optimal solution $w^\star$ by employing a  sequence of step lengths $\{ \alpha_k \}$ that converge to zero according to the schedule  proposed by \cite{robbins1951stochastic}. However, that provides only a sub-linear rate of convergence, which is of little interest in our context where large batches are employed and  some type of linear convergence is expected.  In this light, Theorem~\ref{thm:const} is more relevant to practice; nonetheless, we state the theorem here for completeness, and, for brevity, refer the reader to \citep[Theorem~3.2]{Sammy_SQN} for more details and the proof.
 
\begin{theorem}	\label{thm:dim}
Suppose that Assumptions A.1-A.4 hold, and let $F^{\star} = F(w^{\star})$, where $w^{\star}$ is the minimizer of $F$. Let $\{w_k\}$ be the iterates generated by Algorithm~\ref{alg:multi} with 
\begin{align*} 
\alpha_k = \frac{\beta}{k+1} \qquad \text{and } \qquad \beta > \frac{\mu_1}{\mu_2^2 \eta \lambda},
\end{align*}
starting from $w_0$. Then for all $k\geq 0$,
\begin{align*}  
\mathbb{E}\left[ F(w_k) - F^{\star}\right] \leq\frac{Q(\beta)}{k+1}, 
\end{align*}
where $
Q(\beta) = \max \left\{ \frac{\mu_2^2\beta^2\gamma^2\Lambda}{2(2\mu_1\lambda\beta-1)}, F(w_0) - F^{\star}\right\}.$
\end{theorem} 

Theorem \ref{thm:dim} shows that, for strongly convex functions, the multi-batch L-BFGS method with an appropriate schedule of diminishing step lengths converges to the optimal solution at a sub-linear rate. We should mention that another way to establish convergence to the optimal solution for the multi-batch L-BFGS method is to employ variance reduced gradients \citep{Schmidt2016,johnson2013accelerating,
defazio2014saga,nguyen2017sarahicml,konevcny2016mini,
nguyen2017stochastic,lin2015universal}. In this setting, one can establish linear convergence to the optimal solution using constant step lengths. We defer the analysis of the multi-batch L-BFGS method that employs variance reduced gradients to a different study \citep{berahas2017varredmbLBFGS}. 

\subsection{Nonconvex case}
The BFGS method is known to fail on nonconvex problems \citep{mascarenhas2004bfgs,dai2002convergence}. Even for L-BFGS, which makes only a finite number of updates at each iteration, one cannot guarantee that the Hessian approximations have eigenvalues that are uniformly bounded above and away from zero. To establish convergence of the (L-)BFGS method in the nonconvex setting several techniques have been proposed including \emph{cautious} updating \citep{li2001global}, \emph{modified} updating \citep{li2001modified} and \emph{damping} \citep{powell1978algorithms}. Here we employ a cautious strategy that is well suited to our particular algorithm; we skip the Hessian update, i.e., set $H_{k+1} = H_k$,  if the curvature condition  
\begin{align} \label{curv}
	y_k^Ts_k \geq {\epsilon} \| s_k \|^2
\end{align}
 is not satisfied, where $\epsilon>0$ is a predetermined constant. On the other hand, sufficient curvature is guaranteed when the updates are not skipped. Using said mechanism, we show that the eigenvalues of the Hessian matrix approximations generated by the multi-batch L-BFGS method are bounded above and away from zero (Lemma \ref{lemma3}). The analysis presented in this section is based on the following assumptions.
 
  \paragraph*{\textbf{Assumptions B}}
 %\textbf{Assumptions B.}
\emph{
\begin{enumerate}
\item $F$ is twice continuously differentiable.
\item The gradients of $F$ are $\Lambda$-Lipschitz continuous for all $w \in \mathbb{R}^d$; the gradients of $F^{S}$ are $\Lambda_{S}$-Lipschitz continuous for all $w \in \mathbb{R}^d$ and all sets $S \subset  \{1,2,\ldots,n\}$ of length $|S|=r \cdot n$; and, the gradients of $F^{O}$ are $\Lambda_{O}$-Lipschitz continuous for all $w \in \mathbb{R}^d$ and all sets $O \subset  \{1,2,\ldots,n\}$ of length $|O| = o\cdot r\cdot n$.%; and the gradients of the component function $f_i$ are $\Lambda_i$-Lipschitz continuous for all $w \in \mathbb{R}^d$ and all $i\in[n]$. % of cardinality  $\left| O\right| = o\cdot r \cdot n$.
\item The function $ F(w)$ is bounded below by a scalar $\widehat F$ .
\item There exist constants $\gamma \geq 0$ and $\eta \geq 1$ such that 
%\begin{align}    \label{eq:b_grad1}
$\mathbb{E}_{S}\left[ \| \nabla  F^{S}(w) \|^2 \right] \leq \gamma^2 + \eta \| \nabla F(w)\|^2$
%\mathbb{E}_{S_k}\Bigg[ \Big|\Big| \nabla  F^{S_k}(w) \Big|\Big| \Bigg]^2 \leq \gamma^2.
%\end{align}
for all $w \in \mathbb{R}^d$ and all sets $S\subset  \{1,2,\ldots,n\}$  of length $|S|=r \cdot n$. % of cardinality $\left| S\right| = r \cdot n$.
\item The samples $S$ are drawn independently and $\nabla F^{S}(w)$ is an unbiased estimator of the true gradient $\nabla F(w)$ for all $w \in \mathbb{R}^d$, i.e.,
%\begin{align}		\label{eq:unbiased}
$\mathbb{E}_{S} [ \nabla F^{S}(w) ] = \nabla F(w).
$
%\end{align}
\end{enumerate}
}

Similar to the strongly convex case, we first show that the eigenvalues of the L-BFGS Hessian approximations are bounded above and away from zero.

\begin{lemma}		\label{lemma3}
Suppose that Assumptions B.1 \& B.2  hold. Let $\{H_k \}$ be the inverse Hessian approximations generated by Algorithm~\ref{alg:multi}, with the modification that the inverse Hessian approximation update is performed only when  \eqref{curv} is satisfied,
%\begin{align*} 	%\label{eq:skip}		
%	y_k^Ts_k \geq {\epsilon} \| s_k \|^2,
%\end{align*}
for some $\epsilon >0$, else $H_{k+1} = H_{k}$. Then, there exist constants $0<\mu_1\leq \mu_2$ such that 
\begin{align*}    %\label{eq:lemma}
\mu_1 I \preceq H_k \preceq \mu_2 I,\qquad \text{for } k=0,1,2,\dots 
\end{align*}
\end{lemma}

\begin{proof}
Similar to the proof of Lemma \ref{lemma1}, we study the direct Hessian approximation $B_k = H_k^{-1}$. The curvature pairs $s_k$ and $y_k$ are updated via the following formulae
\begin{align*}   % \label{eq:curv_upd1}
y_{k+1}=g_{k+1}^{O_{k}}-g_{k}^{O_{k}}, \qquad s_{k+1} = w_{k+1}-w_k.
\end{align*}
The skipping mechanism \eqref{curv} provides both an upper and lower bound on the quantity $\frac{\|y_k\|^2 }{y_k^Ts_k}$, which in turn ensures that the initial L-BFGS Hessian approximation is bounded above and away from zero. The lower bound is attained by repeated application of Cauchy's inequality to condition \eqref{curv}. We have from \eqref{curv} that
\begin{align*}
	\epsilon \| s_k \|^2 &\leq y_k^Ts_k \leq  \| y_k \| \| s_k \| \quad \Rightarrow \quad \| s_k \| \leq \frac{1}{\epsilon} \| y_k \|,
\end{align*} 
from which it follows that
\begin{align}		\label{eq:lower}
	s_k^Ty_k \leq \| s_k \| \| y_k \| \leq \frac{1}{\epsilon} \| y_k \|^2 \quad \Rightarrow \quad \frac{\| y_k \|^2}{s_k^Ty_k} \geq \epsilon.
\end{align}
The upper bound is attained by the Lipschitz continuity of sample gradients (Assumption $B.2$),
\begin{align}		\label{eq:upper}
	 y_k^Ts_k  \geq \epsilon \| s_k \|^2  
	\geq  \epsilon \frac{ \| y_k \|^2}{\Lambda_{O}^2} \quad \Rightarrow \quad \frac{\| y_k \|^2}{s_k^Ty_k} \leq \frac{\Lambda_{O}^2}{\epsilon}.
\end{align} 
Combining \eqref{eq:lower} and \eqref{eq:upper},
\begin{align*}
 \epsilon \leq \frac{\|y_k\|^2 }{y_k^Ts_k} \leq \frac{\Lambda_{O}^2}{\epsilon}.
\end{align*}

The above proves that the eigenvalues of the matrices $B_k^{(0)}=\frac{y_k^Ty_k}{s_k^Ty_k}I$ at the start of the L-BFGS update cycles are bounded above and away from zero, for all $k$. The rest of the proof follows the same Trace-Determinant argument as in the proof of Lemma \ref{lemma1}, the only difference being that the last inequality in \eqref{eq:det} comes as a result of the cautious update strategy.
\end{proof}

We now follow the analysis in \citep[Chapter~4]{bottou2018optimization} to establish the following result about the behavior of the gradient norm for the multi-batch L-BFGS method with a cautious update strategy. 

\begin{theorem} \label{thm_non_MB}
Suppose that Assumptions B.1-B.5 hold. Let $\{w_k\}$ be the iterates generated by Algorithm~\ref{alg:multi}, with the modification that the inverse Hessian approximation update is performed only when  \eqref{curv} is satisfied,
%\begin{align*} 	%\label{eq:skip}		
%	y_k^Ts_k \geq {\epsilon} \| s_k \|^2,
%\end{align*}
for some $\epsilon >0$, else $H_{k+1} = H_k$,  where $\alpha_k = \alpha$ satisfies
\begin{align*}	
0 <  \alpha  \leq \frac{\mu_1}{\mu_2^2 \eta \lambda},
\end{align*}
and $w_0$ is the starting point. Then, for all $k\geq0$,
\begin{align*}	%\label{eq:thm2}
\mathbb{E} \left[\frac{1}{\tau}\sum_{k=0}^{\tau-1} \| \nabla F(w_k) \|^2 \right] & \leq \frac{\alpha \mu_2^2 \gamma^2 \Lambda}{ \mu_1 } + \frac{2[ F(w_0) - \widehat{F}]}{\alpha \mu_1 \tau }\\
&  \xrightarrow[]{\tau\rightarrow \infty}\frac{\alpha \mu_2^2 \gamma^2 \Lambda}{ \mu_1 }.
\end{align*}
\end{theorem}

\begin{proof}
Starting with \eqref{eq:proof_2} and taking %an expectation over the batch $S_k$, 
%\begin{align}		\label{eq:proof3.9}
%\mathbb{E}_{S_k}[ F(w_{k+1})] &\leq F(w_k) - \alpha \mu_1 \| \nabla F(w_k) \|^2 + \frac{\alpha^2 \mu_2^2 \Lambda}{2} \mathbb{E}_{S_k} \Big[ \| \nabla F^{S_k}(w_k)\|^2 \Big]  \nonumber\\
%& \leq F(w_k) - \alpha \mu_1 \| \nabla F(w_k) \|^2  + \frac{\alpha^2 \mu_2^2  \Lambda}{2} \left(\gamma^2 + \eta \| \nabla F(w)\|^2 \right) \nonumber \\
%& = F(w_k) - \alpha \left(\mu_1 - \frac{\alpha \mu_2^2  \eta\Lambda}{2}\right) \| \nabla F(w_k) \|^2 + \frac{\alpha^2 \mu_2^2 \gamma^2\Lambda}{2}\\
%& \leq F(w_k) - \frac{\alpha \mu_1}{2} \| \nabla F(w_k) \|^2 + \frac{\alpha^2 \mu_2^2 \gamma^2\Lambda}{2} \nonumber,
%\end{align}
%where the second inequality holds due to Assumption $B.4$, and the fourth inequality is obtained by using the upper bound on the step length. Taking 
the expectation over all  batches $S_0,S_1,...,S_{k-1}$ and all history starting with $w_0$ yields
\begin{align}	\label{eq:nc1}
	\phi_{k+1}-\phi_k \leq - \frac{\alpha \mu_1}{2}  \mathbb{E}\| \nabla F(w_k) \|^2  + \frac{\alpha^2 \mu_2^2 \gamma^2 \Lambda}{2},
\end{align}
where $\phi_k = \mathbb{E}[F(w_k)] $. Summing \eqref{eq:nc1} over the first $\tau$ iterations
\begin{align} \label{eq:nc2}
	\sum_{k=0}^{\tau-1} [\phi_{k+1}-\phi_k] &\leq - \frac{\alpha \mu_1}{2}  \sum_{k=0}^{\tau-1}  \mathbb{E}\| \nabla F(w_k) \|^2 + \sum_{k=0}^{\tau-1} \frac{\alpha^2 \mu_2^2 \gamma^2 \Lambda}{2} \nonumber \\
	&= - \frac{\alpha \mu_1}{2}   \mathbb{E} \left[\sum_{k=0}^{\tau-1} \| \nabla F(w_k) \|^2 \right] +  \frac{\alpha^2 \mu_2^2 \gamma^2 \Lambda \tau}{2}.
\end{align}
The left-hand-side of the above inequality is a telescoping sum
\begin{align*}
	\sum_{k=0}^{\tau-1} \left[\phi_{k+1}-\phi_k \right]  &= \phi_{\tau}-\phi_0 %\nonumber \\
	= \mathbb{E}[F(w_{\tau})] -F(w_0) %\nonumber\\
	\geq \widehat{F} -F(w_0). 
\end{align*}
Substituting the above expression into \eqref{eq:nc2} and rearranging terms
\begin{align*}
	\mathbb{E} \left[\sum_{k=0}^{\tau-1} \| \nabla F(w_k) \|^2 \right] \leq \frac{\alpha \mu_2^2 \gamma^2 \Lambda \tau}{ \mu_1 } + \frac{2[ F(w_0) - \widehat{F}]}{\alpha \mu_1 }.
\end{align*}
Dividing the above equation by $\tau$ completes the proof.
\end{proof}

This result bounds the average norm of the gradient of $F$  after the first $\tau-1$ iterations, and shows that, in expectation, the iterates spend increasingly more time in regions where the objective function has a small gradient. Under appropriate conditions, we can establish a convergence rate for the multi-batch L-BFGS method with cautious updates to a stationary point of $F$, similar to the results proven for the SG method \citep{ghadimi2013stochastic}. For completeness we state and prove the result.

\begin{theorem}	\label{thm:nonconvex_stat} 
Suppose that Assumptions B.1-B.5 hold. Let $\{w_k\}$ be the iterates generated by Algorithm~\ref{alg:multi}, with the modification that the inverse Hessian approximation update is performed only when  \eqref{curv} is satisfied,  
%\begin{align*} 	%\label{eq:skip}		
%	y_k^Ts_k \geq {\epsilon} \| s_k \|^2,
%\end{align*}
for some $\epsilon >0$, else $H_{k+1} = H_k$. Let
\begin{align*}
	\alpha_k = \alpha = \frac{c}{\sqrt{\tau}}, \quad c = \sqrt{\frac{2(F(w_0) - \hat{F})}{\mu_2^2 \gamma^2 \Lambda}}, \quad \delta(\alpha) = \mu_1 - \frac{\alpha \mu_2^2  \eta\Lambda}{2},
\end{align*}
where 
$%\begin{align*}
\tau > %\frac{(F(w_0) - \hat{F})\mu_2^2\eta^2 \Lambda}{2\gamma^2\mu_1^2}=
\frac{c^2\mu_2^4 \eta^2 \Lambda^2}{4\mu_1^2}
$, %\end{align*}
and $w_0$ is the starting point. Then, 
\begin{align*}
	\min_{0 \leq k \leq \tau-1} \mathbb{E}\left[\| \nabla F(w_k)\|^2\right] \leq \sqrt{\frac{2(F(w_0) - \hat{F})\mu_2^2 \gamma^2 \Lambda}{\delta(\alpha)^2 \tau}}.
\end{align*}
\end{theorem}

\begin{proof} 
Starting with \eqref{eq:proof3}, we have
\begin{align}	\label{eq:nonconv1}
\mathbb{E}_{S_k}[ F(w_{k+1})] &\leq F(w_k) - \alpha \left(\mu_1 - \frac{\alpha \mu_2^2  \eta\Lambda}{2}\right) \| \nabla F(w_k) \|^2 + \frac{\alpha^2 \mu_2^2 \gamma^2\Lambda}{2} \nonumber\\
& = F(w_k) - \alpha \delta(\alpha) \| \nabla F(w_k) \|^2 + \frac{\alpha^2 \mu_2^2 \gamma^2\Lambda}{2},
\end{align}
where $\delta(\alpha) = \mu_1 - \frac{\alpha \mu_2^2  \eta\Lambda}{2}$. We require that this quantity is greater than zero, $\delta(\alpha)>0$; this discussion is deferred to the end of the proof.

Taking an expectation over all  batches $S_0,S_1,...,S_{k-1}$ and all history starting with $w_0$, and rearranging \eqref{eq:nonconv1} yields
\begin{align*}
	\mathbb{E}\left[\| \nabla F(w_k) \|^2\right] \leq \frac{1}{\alpha \delta(\alpha)}\mathbb{E}[F(w_k)-F(w_{k+1})] + \frac{\alpha \mu_2^2 \gamma^2 \Lambda}{2 \delta(\alpha)}.
\end{align*}
Summing over $k=0,...,\tau-1$ and dividing by $\tau$
\begin{align*}
\min_{0 \leq k \leq \tau-1} \mathbb{E}\left[\| \nabla F(w_k)\|^2\right] & \leq \frac{1}{\tau}\sum_{k=0}^{\tau-1}\mathbb{E}\left[\| \nabla F(w_k) \|^2\right] \\
& \leq \frac{1}{\alpha \delta(\alpha) \tau}\mathbb{E}[F(w_0)-F(w_\tau)] + \frac{\alpha \mu_2^2 \gamma^2 \Lambda}{2 \delta(\alpha)} \\
& \leq \frac{1}{\alpha \delta(\alpha) \tau}[F(w_0)-\hat{F}] + \frac{\alpha \mu_2^2 \gamma^2 \Lambda}{2 \delta(\alpha)} \\
& \leq \frac{1}{ \delta(\frac{c}{\sqrt{\tau}}) c \sqrt{\tau}}[F(w_0)-\hat{F}] + \frac{c \mu_2^2 \gamma^2 \Lambda}{2 \delta(\frac{c}{\sqrt{\tau}})\sqrt{\tau}}.
\end{align*}
The first inequality holds because the minimum value is less than the average value, and the third inequality holds because $\hat{F} \leq F(x_\tau)$ (Assumption $B.3$). The last expression comes as a result of using the definition of the step length, $\alpha = \frac{c}{\sqrt{\tau}}$. Setting 
\begin{align}	\label{eq:noncon_c}
	c = \sqrt{\frac{2(F(w_0) - \hat{F})}{\mu_2^2 \gamma^2 \Lambda}},
\end{align}
yields the desired result.

We now comment on the quantity $\delta(\alpha)$ that first appears in \eqref{eq:nonconv1}, and that is required to be positive. To ensure that $\delta(\alpha)>0$, the step length must satisfy, $\alpha < \frac{2\mu_1}{\mu_2^2 \eta \Lambda}$. Since the explicit form of the step length is $\alpha = \frac{c}{\sqrt{\tau}}$, where $c$ is \eqref{eq:noncon_c}, we require that 
\begin{align}	\label{eq:noncon_step}
	\alpha =  \frac{c}{\sqrt{\tau}} < \frac{2\mu_1}{\mu_2^2 \eta \Lambda}.
\end{align}
In order to ensure that \eqref{eq:noncon_step} holds, we impose that
\begin{align*}
	\tau > \frac{c^2\mu_2^4 \eta^2 \Lambda^2}{4\mu_1^2} = \frac{(F(w_0) - \hat{F})\mu_2^2\eta^2 \Lambda}{2\gamma^2\mu_1^2}.
\end{align*}
\end{proof}

The result of Theorem \ref{thm:nonconvex_stat} establishes a sub-linear rate of convergence, to a stationary point of $F$, for the multi-batch L-BFGS method on nonconvex objective functions. The result is somewhat strange as it requires a priori knowledge of $\tau$, the total number of iteration. In practice, one would use $\alpha_k = \frac{1}{\sqrt{k}}$, which would result in a $\mathcal{O}(\frac{1}{\sqrt{k}})$ convergence rate.

%%% Numerical Results %%%%
\section{Numerical Results}
\label{sec:num_res}
\setcounter{equation}{0}

We present numerical experiments on several problems that arise in machine learning, such as logistic regression binary classification and neural network training, in order to evaluate the performance of the proposed multi-batch L-BFGS method. The experiments verify that the proposed method is robust, competitive and achieves a good balance between computation and communication in the distributed setting. In Section \ref{sec:log_reg}, we evaluate the performance of the multi-batch L-BFGS method on binary classification tasks in both the multi-batch and fault-tolerant settings. In Section \ref{sec:neural_nets}, we demonstrate the performance of the multi-batch L-BFGS method on neural network training tasks,
%using several datasets and using different neural network architectures
 and compare against some of the state-of-the-art methods. Finally, in Section \ref{sec:scaling}, we illustrate the strong and weak scaling properties of the multi-batch L-BFGS method.

\subsection{Logistic Regression}
\label{sec:log_reg}
In this section, we focus on logistic regression problems; the optimization problem can be stated as:
\begin{align*}
	\min_{w \in \mathbb{R}^d} F(w) =  \frac{1}{n}\sum_{i=1}^{n}\log\left(1+e^{-y^i(w^Tx^i)}\right)
  + \frac{\sigma}{2} \|w\|^2,
\end{align*}
where $ (x^i, y^i)_{i=1}^n$  denote the training examples and $\sigma = \frac1n$ is the regularization parameter. 

We present numerical results that evaluate the performance of the proposed robust multi-batch L-BFGS scheme  (Algorithm \ref{alg:multi}) in both the \emph{multi-batch} (Figure \ref{fig:demo:MB}) and \emph{fault tolerant} (Figure \ref{fig:demo:FT}) settings, on the \texttt{webspam} dataset\footnote{LIBSVM: \url{https://www.csie.ntu.edu.tw/~cjlin/libsvmtools/datasets/binary.html}}. We compare our proposed method (Robust L-BFGS) against three methods: $(i)$ multi-batch L-BFGS without enforcing sample consistency (L-BFGS), where gradient differences are computed using different samples, i.e., $y_k = g^{S_{k+1}}_{k+1} - g^{S_k}_k$; $(ii)$ multi-batch gradient descent (Gradient Descent), which is obtained by setting $H_k = I$ in Algorithm 1; and, $(iii)$ serial SGD (SGD), where at every iteration one sample is used to compute the gradient. We run each method with 10 different random seeds, and, where applicable, report results for different batch ($r$) and overlap ($o$) sizes. In Figures \ref{fig:demo:MB} and \ref{fig:demo:FT} we show the evolution of the norm of the gradient in terms of epochs.

\begin{figure}
\centering
\includegraphics[width=\textwidth]{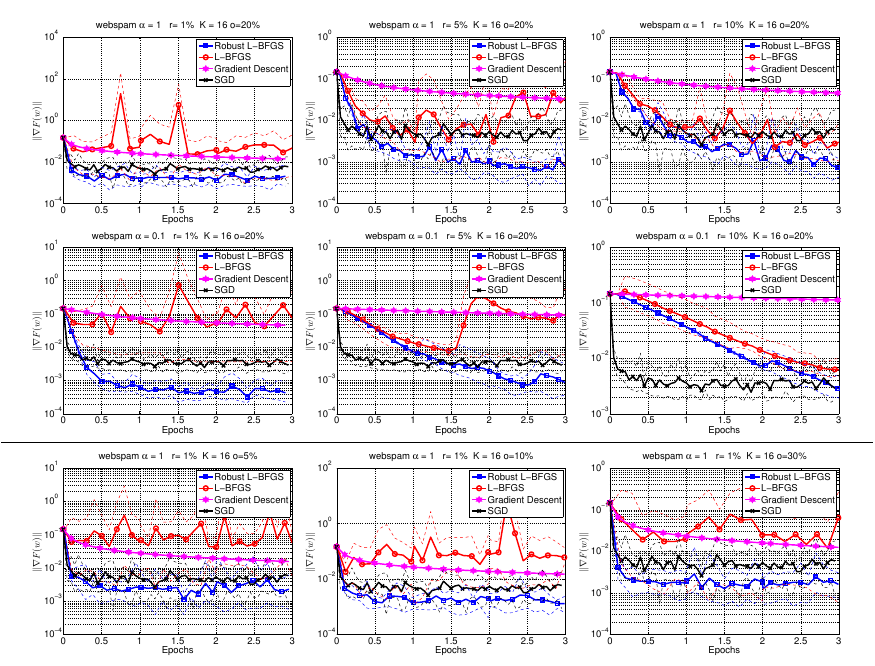}
\caption{\textbf{webspam dataset}. Comparison of Robust L-BFGS, L-BFGS (multi-batch L-BFGS without enforcing sample consistency), Gradient Descent (multi-batch Gradient method) and SGD. Top part:
we used $\alpha \in \{1, 0.1\}$,
$r\in \{1\%,  5\%,  10\%\}$ and $o=20\%$.
Bottom part: we used $\alpha=1$, $r=1\%$ and
$o\in \{5\%,  10\%, 30\%\}$. Solid lines show average performance, and dashed lines show worst and best performance, over 10 runs (per algorithm). $K=16$ MPI processes.}
\label{fig:demo:MB}
\end{figure}

In the multi-batch setting, the proposed method is more stable than the standard L-BFGS method; this is especially noticeable when $r$ is small. On the other hand, serial SGD achieves similar accuracy as the robust L-BFGS method and at a similar rate, at the cost of $n$ communications per epoch versus $\frac{1}{r(1-o)}$ %\todo{This is what we had $100/(1-o)$, I think it should be $\frac{1}{r(1-o)}$} 
% Yes, you are correct ;)
communications per epoch. Figure \ref{fig:demo:MB} also indicates that the robust L-BFGS method is not too sensitive to the size of the overlap. Similar behavior was observed on other datasets, in regimes where $r\cdot o$ was not too small; see \cite[Section A.1]{berahas2017robust_supp}.%Appendix \ref{sec:extnumres_mb}. 

\begin{figure}
\centering

\includegraphics[width=\textwidth]{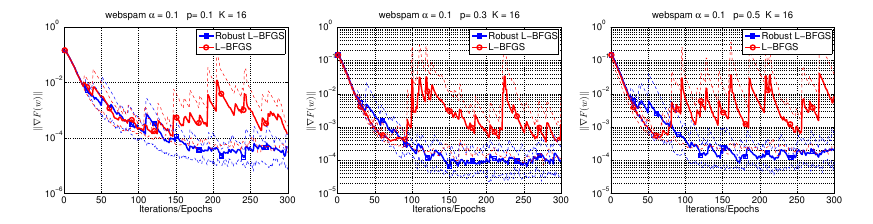}
 
\caption{\textbf{webspam dataset}. Comparison of Robust L-BFGS and L-BFGS in the presence of faults.
We used $\alpha=0.1$ and $p\in \{0.1, 0.3, 0.5\}$. Solid lines show average performance, and dashed lines show worst and best performance, over 10 runs (per algorithm). $K=16$ MPI processes.
}\label{fig:demo:FT}
\end{figure}

Figure  \ref{fig:demo:FT} shows a comparison of the proposed robust multi-batch L-BFGS method and the multi-batch L-BFGS method that does not enforce sample consistency (L-BFGS) in the presence of faults. In these experiments, $p$ denotes the probability that a single node (MPI process) will not return a gradient evaluated on local data within a given time budget. We illustrate the performance of the methods for $\alpha=0.1$ and $p\in \{0.1, 0.3, 0.5\}$. We observe that the robust implementation is not affected much by the failure probability $p$. Similar behavior was observed on other datasets; see \cite[Section A.2]{berahas2017robust_supp}.% Appendix \ref{sec:extnumres_ft}.

 \subsection{Neural Networks}
\label{sec:neural_nets}
In this section, we study the performance of the multi-batch L-BFGS method\footnote{Code available at: \url{https://github.com/OptMLGroup/Multi-Batch_L-BFGS}.} on Neural Network tasks, on the MNIST and CIFAR10/CIFAR100 datasets\footnote{MNIST available at: \url{http://yann.lecun.com/exdb/mnist/}. CIFAR10/CIFAR100 available at: \url{https://www.cs.toronto.edu/~kriz/cifar.html}.}. Table \ref{tbl:NN} summarizes the network architectures that we used; see \cite[Section B]{berahas2017robust_supp} % Appendix \ref{sec:extnumres_nn} 
for more details. The first problem is convex; all other problems are nonconvex.

%\footnote{Available at: \url{http://yann.lecun.com/exdb/mnist/}.}
%\footnote{Available at: \url{https://www.cs.toronto.edu/~kriz/cifar.html}.}
%The first network ({\it DNN}) is a two-layer Convolutional Neural Network with cross-entropy loss, and the second network ({\it Regression}) is a Feed-Forward Neural Network with no hidden units and $\ell_2$ loss function. The second problem is convex. 

\begin{table}[h!]

\caption{ Structure of Neural Networks.}
\small
\vskip5pt
\label{tbl:NN}
\centering
\begin{tabular}{lccrc}
\toprule
\textbf{Network} &
 \textbf{Type} &
 \textbf{\# of layers} & 
  \textbf{$d$} & \textbf{Ref.}  \\  \midrule

\textbf{MNIST MLC} & fully connected (FC)& 1& 7.8k & \cite{lecun1998gradient}
 \\ \hdashline
\textbf{MNIST DNN (SoftPlus)} & conv+FC  & 4& 1.1M &  \cite{PyTorchExamples}
\\ \hdashline
\textbf{MNIST DNN (ReLU)} & conv+FC & 4& 1.1M & \cite{PyTorchExamples} 
 \\ \hdashline
\textbf{CIFAR10 LeNet} & conv+FC & 5& 62.0k &  \cite{lecun1998gradient}
\\ \hdashline
\textbf{CIFAR10 VGG11} &
conv+batchNorm+FC
 & 29& 9.2M & \cite{simonyan2014very}
\\ \hdashline
\textbf{CIFAR100 VGG11} & conv+batchNorm+FC& 29& 9.2M & \cite{simonyan2014very}
\\
 \bottomrule
{\small MLC =  Multiclass Linear Classifier}
\end{tabular}
\end{table}

We implemented our algorithm in PyTorch \citep{paszke2017automatic} and compare against popular and readily available algorithms: $(i)$ SGD \cite{robbins1951stochastic}, and $(ii)$  Adam \citep{kingma2014adam}. We denote our proposed method as LBFGS in the figures in this section. Note, we implemented our method with the \emph{cautious updating} strategy, and For each method, we conducted a grid search to find the best learning rate $\alpha \in \{2^{0},2^{-1},\dots,  2^{-10}\}$, and also investigated the effect of different batch sizes $| S | \in \{50, 100, 200, 500, 1000, 2000, 4000 \}$; see \cite[Section B]{berahas2017robust_supp} % Appendix \ref{sec:extnumres_nn}
 for detailed experiments with all batch sizes. For the multi-batch L-BFGS method we also investigated the effect of history length $m \in \{1,2,5,10,20 \}$. The overlap used in our proposed method was $20\%$ of the batch, $o = 0.20$.

%\textcolor{red}{The authors in \cite{keskar2016adaqn} observed that the widely used Barzilai-Borwein-type scaling $\frac{s_k^T y_k}{y_k^T y_k} I$ of the initial Hessian approximation may lead to quasi-Newton updates that are not stable when small batch sizes are employed, especially for deep neural training tasks, and as such propose an Agadrad-like scaling of the initial BFGS matrix. To obviate this instability, we implement a variant of the multi-batch L-BFGS method (LBFGS2) in which we scale the initial Hessian approximations as $\alpha I$. }

The authors in \cite{keskar2016adaqn} observed that the widely used Barzilai-Borwein-type scaling $\frac{s_k^T y_k}{y_k^T y_k} I$ of the initial Hessian approximation may lead to quasi-Newton updates that are not stable when small batch sizes are employed, especially for deep neural training tasks, and as such propose an Agadrad-like scaling of the initial BFGS matrix. To obviate this instability, we implement a variant of the multi-batch L-BFGS method (LBFGS2) in which we scale the initial Hessian approximations as $\alpha I$. We ran experiments with both scaling strategies and the overall results were similar. Therefore, in the figures in this section we only show results for the latter strategy. 

\begin{figure}
\centering

\includegraphics[width=\textwidth]{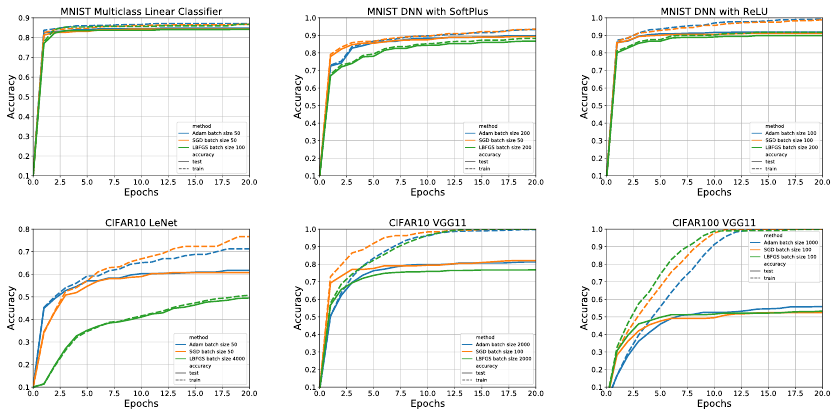}

\caption{Comparison of multi-batch L-BFGS (LBFGS), Adam and SGD. Evolution of the training and testing accuracy for the best parameter setting for each method and problem. Top right: MNIST Multiclass Linear Classifier; Top middle: MNIST DNN (SoftPlus); Top Right: MNIST DNN (ReLU); Bottom left: CIFAR10 LeNet; Bottom middle: CIFAR10 VGG11; Bottom right: CIFAR100 VGG11. \label{fig:best_all}}
\end{figure}

Figure \ref{fig:best_all} illustrates the evolution of the \textcolor{black}{running maximum of the} training accuracy (dashed lines) and testing accuracy (solid lines) for the best parameter settings for each method over the first 20 epochs of training. \textcolor{black}{By best parameter settings we mean the run that achieved the highest training accuracy within 20 epochs.} One can make several observations. First, it appears that the multi-batch L-BFGS method is competitive with the first-order methods on all training problems, except for CIFAR10 LeNet, in terms of both training and testing accuracy. Second, for \textcolor{black}{half of the problems (three out of six)}, the best runs of the multi-batch L-BFGS method use larger batch sizes than the first-order methods. \textcolor{black}{Of course, this benefit is not as clear on the neural network training problems as it is in the logistic regression problems. } Third, the benefits of incorporating second-order information are not as apparent in these nonconvex problems as compared to the problems presented in Section \ref{sec:log_reg}. We attribute this to two things: (1) nonconvex problems are hard, and (2) quasi-Newton methods are better are capturing curvature information for convex problems.

\begin{figure}
\centering

\includegraphics[width=\textwidth]{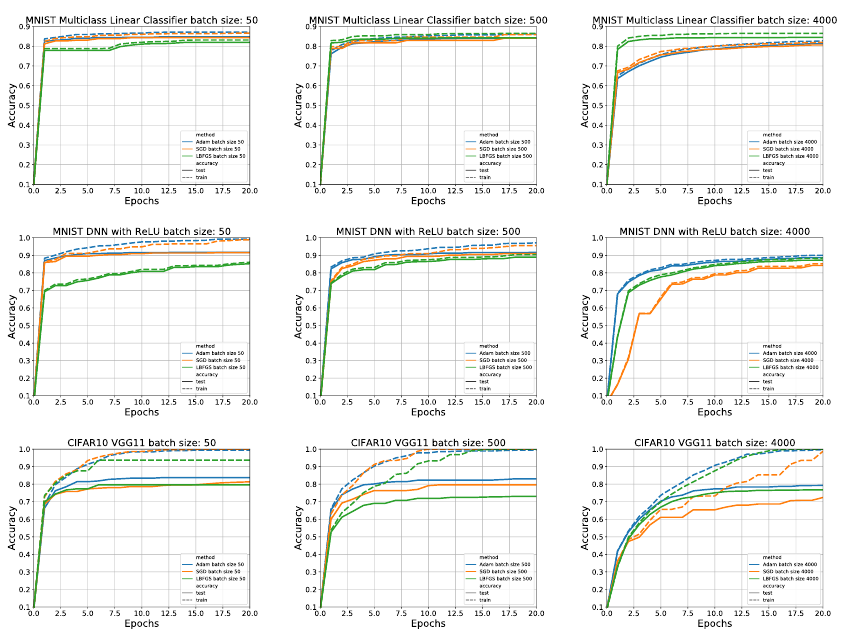}

\caption{Comparison of multi-batch L-BFGS (LBFGS), Adam and SGD. Evolution of the training and testing accuracy for different batch sizes ($| S | \in \{50, 500, 4000 \}$). Top row: MNIST Multiclass Linear Classifier; Middle row: MNIST DNN (ReLU); Bottom row: CIFAR10 VGG11. \label{fig:diff_batches}}
\end{figure}

We now investigate the effect of the batch size. In Figure \ref{fig:diff_batches} we show the evolution of the \textcolor{black}{running maximum of the} training/testing accuracy for different batch sizes $| S | \in \{50, 500, 4000 \}$ for three of the problems. For a complete set of results, see \cite[Section B]{berahas2017robust_supp}. % Appendix \ref{sec:extnumres_nn}. 
Overall, one can observe that for small batch sizes, the multi-batch L-BFGS variants perform worse than the first order methods. However, when large batches are employed (a regime that is favorable foe GPU computing), the multi-batch L-BFGS method performs on par with the other methods. \textcolor{black}{Moreover, it appears that on several problems the performance of the multi-batch L-BFGS method is less affected by the size of the batch, i.e., the variability in the final training and testing error (after 20 epochs) in terms of batch size is smaller for the multi-batch L-BFGS method than for the stochastic first-order methods; see also \cite[Section B]{berahas2017robust_supp}.}

In order to understand why the multi-batch L-BFGS method does not perform well for small batches, we looked at two diagnostic measures: $(i)$ the angle between the true gradient curvature vector $y_d$ and subsampled gradient curvature vector $y_s$ ($\frac{\langle y_s,y_d \rangle}{\| y_s\| \| y_d \|}$); and $(ii)$ the ratio of subsampled gradient curvature vector to true gradient curvature vector ($\frac{y_s}{y_d}$). These measures indicate how informative the curvature information captured by the multi-batch L-BFGS method really is. Values close to $1$ (dashed red lines) are ideal for both measures. We chose 3 different points (the starting point, a point after 3 epochs of Adam, and a point after 10 epochs of Adam). From those points we took a gradient descent step with sufficiently small step length, and computed the true gradient curvature vector ($y_d$). We also computed 100 different stochastic variants of the gradient curvature vector ($y_s$) using different batch sizes ($| S | \in \{50, 100, 200, 500, 1000, 2000, 4000 \}$), and calculated the values of the two metrics. We illustrate the results in Figure \ref{fig:angle_ratio}; see \cite[Section B]{berahas2017robust_supp} %Appendix \ref{sec:extnumres_nn} 
for more results. Several observations can be drawn from this figure. First, not surprisingly, the metrics improve (get closer to $1$) as the batch size increases. Second, for the convex case, the metrics perform as expected both close and far from the solution; as a result (sufficiently) good curvature information is captured and the method performs well. On the other hand, for the nonconvex problems, the metrics indicate that, especially for small batch sizes, the curvature information captured can be terrible.

\begin{figure}
\centering

\includegraphics[width=\textwidth]{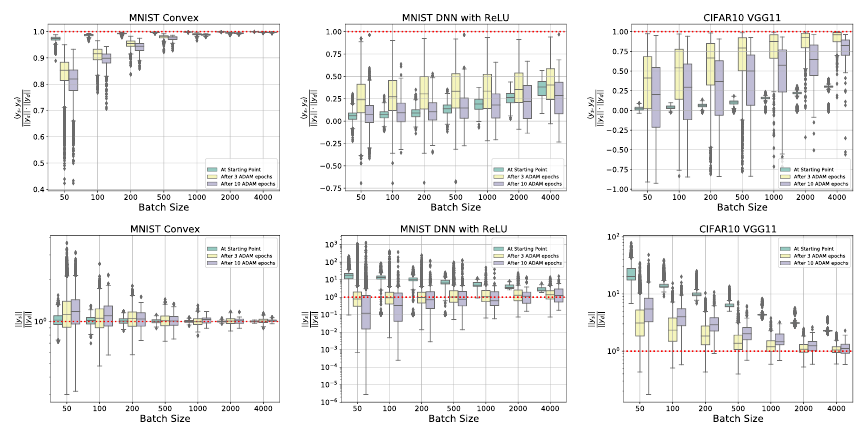}

\caption{Multi-batch L-BFGS diagnostic metrics for different batch sizes $| S | \in \{50, 100, 200, 500, 1000, 2000, 4000 \}$. Top row: angle between the true gradient curvature vector $y_d$ and subsampled gradient curvature vector $y_s$ ($\frac{\langle y_s,y_d \rangle}{\| y_s\| \| y_d \|}$); Bottom row: ratio of subsampled gradient curvature vector to true gradient curvature vector ($\frac{y_s}{y_d}$). Left column: MNIST Multiclass Linear Classifier; Middle column: MNIST DNN (ReLU); Right column: CIFAR10 VGG11.}
\label{fig:angle_ratio}
\end{figure}

\begin{wrapfigure}{r}{0.5\textwidth}
\vspace{-0.75cm}
  \begin{center}
    \includegraphics[width=0.48\textwidth]{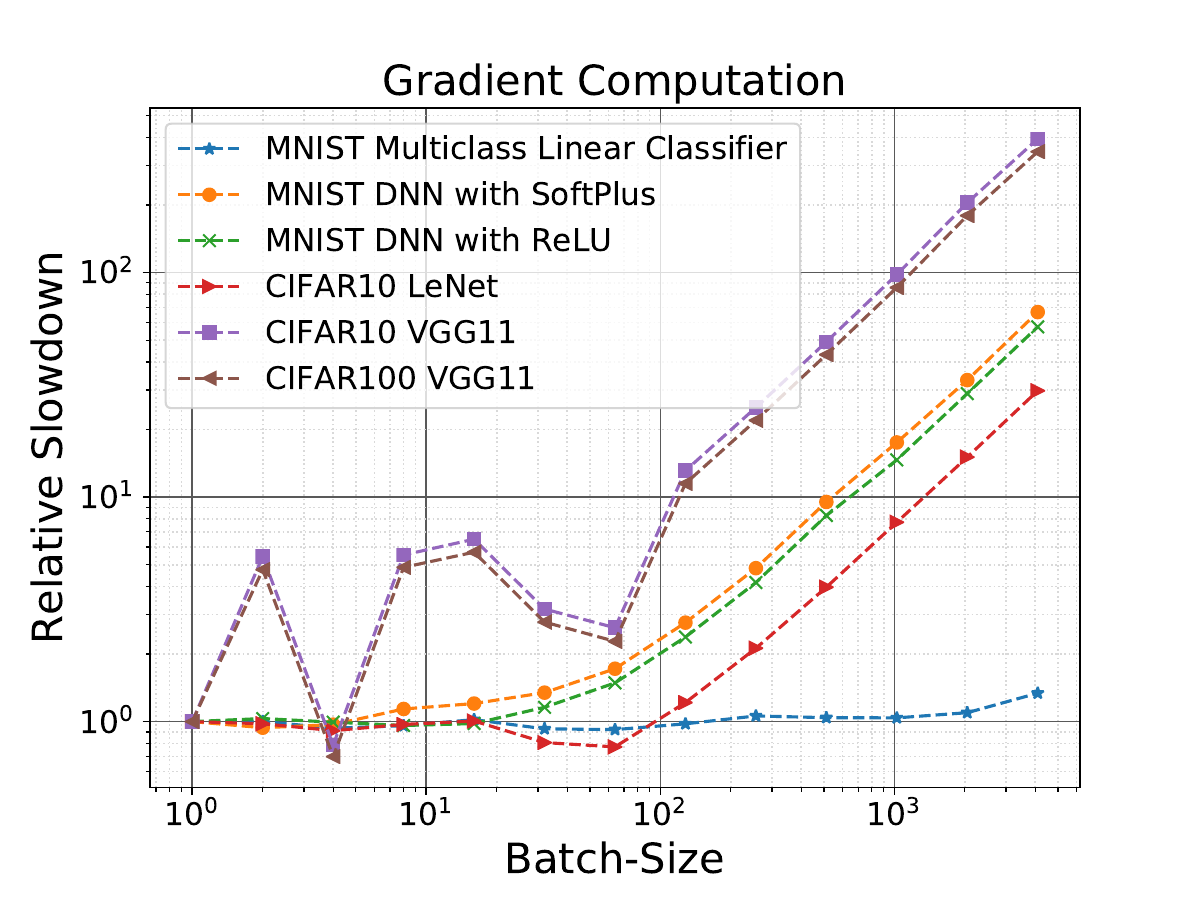}
  \end{center}
\caption{Relative slow-down of computational time to compute gradients for different batch-sizes compared to the computational time to compute gradients with batch-size 1.
}
\label{fig:batchSize1}
\end{wrapfigure}
We should note that using a large batch size is not a bottleneck for current high-performance computing hardware, on the contrary, using small batch sizes leads to under-utilization of computational hardware and in fact hinders the ability to parallelize the methods. Figure \ref{fig:batchSize1} illustrates this phenomenon; we show the average computational time of a gradient evaluation (over 1,000 evaluations) for different batch sizes ranging from $1$ to $4096$ relative to the computational time of the computation of the gradient using a single data point. The computation was performed on an NVIDIA Tesla K80 GPU using PyTorch. It is clear that for the {\it Multiclass Linear classification} task, the compute time of a single gradient is roughly the same as the compute time of a gradient based on a batch of size $1024$, whereas for the larger training tasks, the compute time of the gradients appear constant up to batch sizes of roughly $128$. \textcolor{black}{We should  note however that their is a risk of decreased generalization when increasing the batch size, unless other strategies such as modifying the step size or regularization are used; see e.g., \cite{keskar2016large,goyal2017accurate}.}

\newpage

\subsection{Scaling of the Multi-Batch L-BFGS Implementation}
\label{sec:scaling}

In this section, we study the strong and weak scaling properties of the robust multi-batch L-BFGS method on artificial data. For various values of batch size ($r$) and nodes ($K$), we measure the time needed to compute a gradient (Gradient) and the time needed to compute and communicate the gradient (Gradient$+C$), as well as, the time needed to compute the L-BFGS direction (L-BFGS) and the associated communication overhead (L-BFGS$+C$). \textcolor{black}{The function of which we are computing the gradient is logistic regression. The L-BFGS direction is computed using the Vector-Free L-BFGS implementation \citep{chen2014large}. We should note that the time to compute the gradient, which of course is required for computing the L-BFGS direction, is not included in L-BFGS and L-BFGS$+C$. We report the extra time to compute the L-BFGS step, after having computed the gradient. Thus, the goal of this section is to show that the time needed to compute the L-BFGS direction is insignificant compared to the cost of computing the gradient, which is needed in any case to run first-order methods. }

\subsubsection*{Strong Scaling}

\begin{wrapfigure}{r}{0.4\textwidth}
\vspace{-0.75cm}
  \begin{center}
    \includegraphics[width=0.38\textwidth]{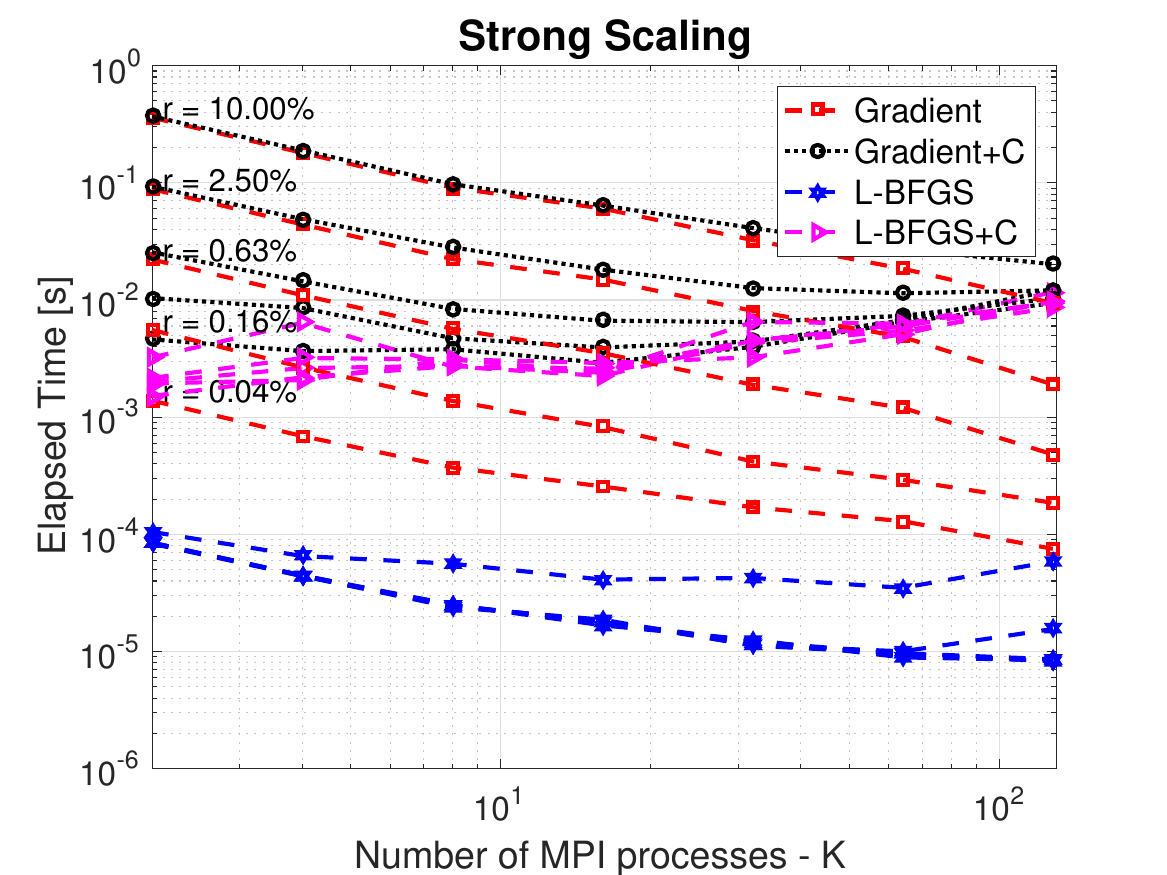}
  \end{center}
  \vspace{-0.25cm}
\caption{Strong scaling of robust multi-batch L-BFGS on a problem with artificial data; 
  $n=10^7$ and $d=10^4$. Each sample has $160$ non-zero elements (dataset size 24GB). \label{strongscaling}}
\end{wrapfigure}
Figure \ref{strongscaling} depicts the strong scaling properties of the multi-batch L-BFGS method, for different batch sizes ($r$) and nodes ($K=1,2,...,128$). For this task, we generate a dataset with $n=10^7$ samples and $d=10^4$ dimensions, where each sample has 160 randomly chosen non-zero elements (dataset size 24GB). One can observe that as the number of nodes ($K$) is increased, the compute times for the gradient and the L-BFGS direction decrease. However, when communication time is considered, the combined cost increases slightly as $K$ is increased. Notice that for large $K$, even when $r=10\%$ (i.e., $10\%$ of all samples processed in one iteration, $\sim$18MB of data), the amount of local work is not sufficient to overcome the communication cost.

\subsubsection*{Weak Scaling --  Fixed Problem Dimension, Increasing Data Size}

\begin{wrapfigure}{r}{0.4\textwidth}
\vspace{-0.75cm}
  \begin{center}
    \includegraphics[width=0.38\textwidth]{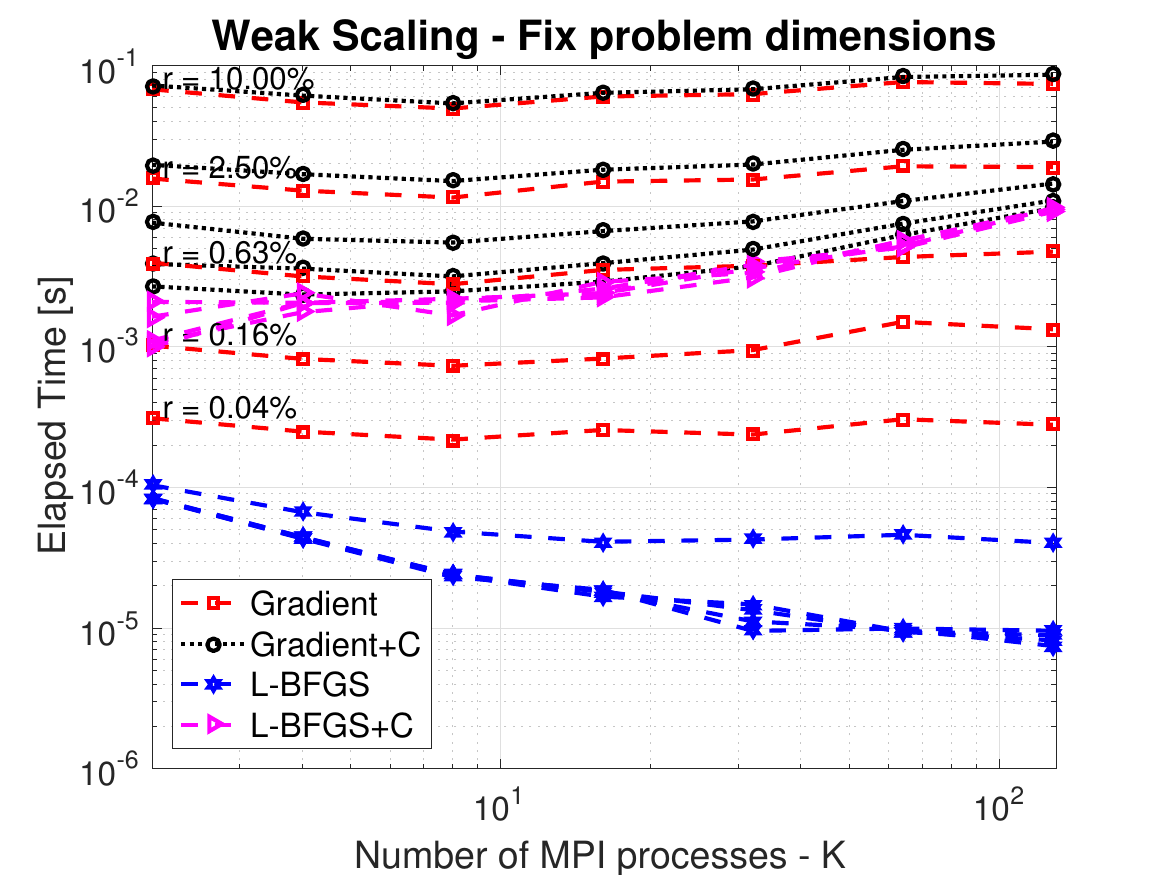}
  \end{center}
  \vspace{-0.25cm}
\caption{Weak scaling of robust multi-batch L-BFGS on a problem with artificial data;  
  $n=10^7$ and $d=10^4$. Each sample has $10\cdot K$ non-zero elements (size of local problem 1.5GB). \label{weakscaling2}}
\end{wrapfigure}
In order to illustrate the weak scaling properties of the algorithm, we generate a data-matrix $X \in \mathbb{R}^{n	\times d}$ ($n=10^7$, $d = 10^4$), and compute the gradient and the L-BFGS direction on a shared cluster with different number of MPI processes ($K=1,2,...,128$). Each sample has $10\cdot K$ non-zero elements, thus for any $K$ the size of local problem is roughly $1.5$GB (for $K=128$ size of data 192GB). Effectively, the dataset size ($n$) is held fixed, but the sparsity of the data decreases as more MPI processes are used. The compute time for the gradient is almost constant, this is because the amount of work per MPI process (rank) is almost identical; see Figure \ref{weakscaling2}. On the other hand, because we are using a Vector-Free L-BFGS implementation \citep{chen2014large} for computing the L-BFGS direction, the amount of time needed for each node to compute the L-BFGS direction decreases as $K$ is increased. However, increasing $K$ does lead to larger communication overhead, and as such the overall time needed to compute and communicate the L-BFGS direction increases slightly as $K$ is increased. For $K=128$ (192GB of data) and $r=10\%$, almost 20GB of data are processed per iteration in less than 0.1 seconds, which implies that one epoch would take around 1 second.

\subsubsection*{Increasing Problem Dimension, Fixed Data Size and $K$}

\begin{wrapfigure}{r}{0.4\textwidth}
\vspace{-1cm}
  \begin{center}
    \includegraphics[width=0.38\textwidth]{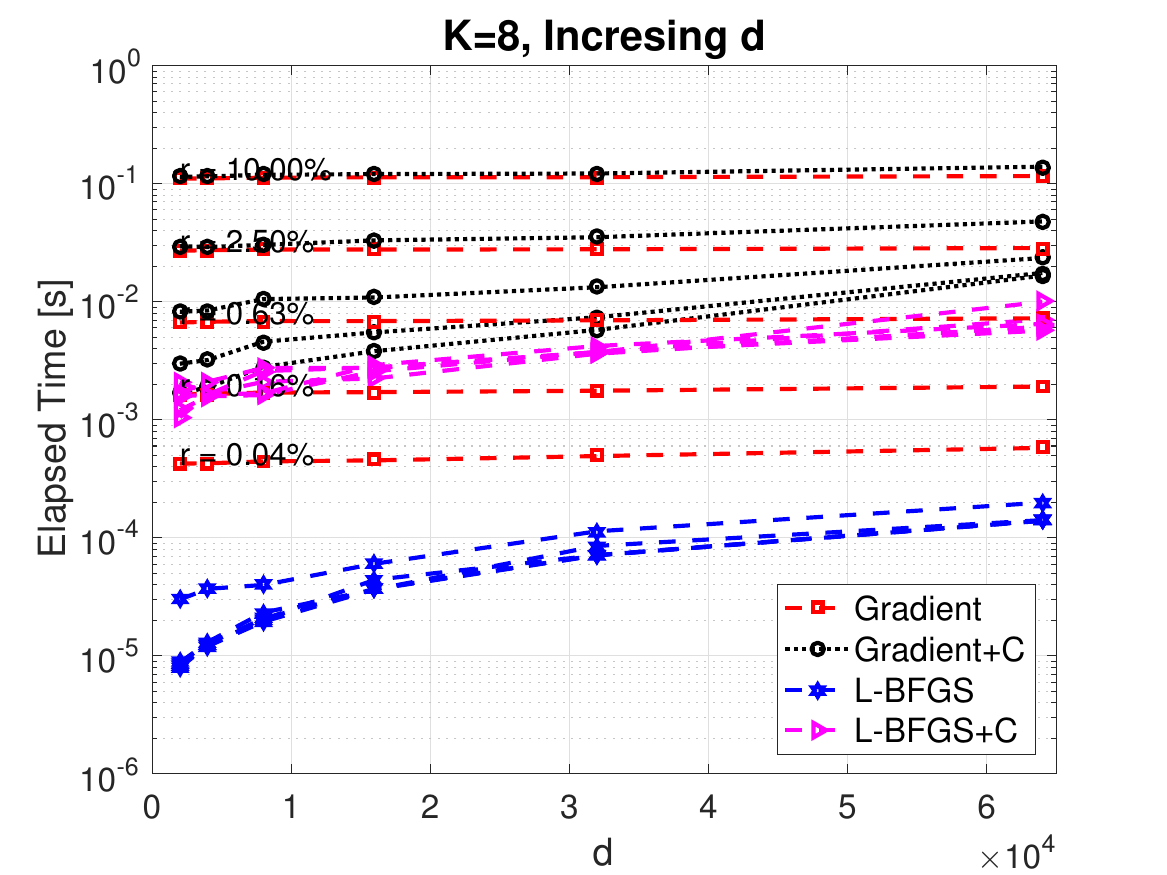}
  \end{center}
    \vspace{-0.25cm}
\caption{Scaling of robust multi-batch L-BFGS on a problem with artificial data; $n=10^7$, increasing $d$ and $K=8$ MPI processes. Each sample had 200 non-zero elements (dataset size 29GB). \label{weakscaling}}
\end{wrapfigure}
In this experiment, we investigate the effect of a change in the dimension ($d$) of the problem on the computation of the gradient and the L-BFGS direction. We fix the size of data (29GB) and the number of MPI processes ($K=8$), and generate data with $n=10^7$ samples, where each sample has 200 non-zero elements. Figure \ref{weakscaling} shows that increasing the dimension $d$ has a mild effect on the computation time of the gradient, while the effect on the time needed to compute the L-BFGS direction is more apparent. However, if communication time is taken into consideration, the time required for the gradient computation and the L-BFGS direction computation increase as $d$ is increased. \textcolor{black}{We should note that the results presented in Figure \ref{weakscaling} are not surprising; there is
minimal change in performance (in terms of the gradient computation) as dimension increases, since the number of nonzero elements is fixed and sparse matrix operations are emplyed.}

\color{black}

\section{Final Remarks}
\label{sec:final_rem}
\setcounter{equation}{0}

In this paper, we assumed that sample consistency is not possible (\emph{fault-tolerant} setting) or desirable (\emph{multi-batch} setting), and described a novel and robust variant of the L-BFGS method designed to deal with two adversarial situations. The success of the algorithm relies on the fact that gradient differences need not be computed on the full batch, rather a small subset can be used alleviating the need for double function evaluations while still maintaining useful curvature information. The method enforces a small degree of control in the sampling process and avoids the pitfalls of using inconsistent gradient differences by performing quasi-Newton updating on the overlap between consecutive samples. 

Our numerical results indicate that provided the overlap is not too small, the proposed method is efficient in practice on machine learning tasks such as binary classification logistic regression and neural network training. The experiments presented in this paper show that the empirical performance of the method matches that predicted by the theory for both strongly convex and nonconvex functions. Specifically, in the strongly convex case the multi-batch L-BFGS method with a constant step length converges to a neighborhood of the solution at a linear rate, and in the nonconvex case the iterates produced by the multi-batch L-BFGS method converge to a neighborhood of a stationary point.

\textcolor{black}{Of course, the development, both theoretical and practical, of stochastic quasi-Newton methods is far from complete, and there are many interesting directions that can and should be investigated. Theoretical analysis that would suggest the batch size and overlap size would be of great interest in practice. Moreover, an investigation of the multi-batch L-BFGS method that employs variance reduced gradients in lieu of the stochastic gradients could have both theoretical and practical advantages. Finally, a stochastic line search that could work in conjunction with the multi-batch L-BFGS method would be novel both algorithmically and theoretically, and would most probably make the method even more competitive in practice.}

\subsubsection*{Acknowledgements} This work was partially supported by DARPA Lagrange award HR-001117S0039 and U.S. National Science Foundation, under award numbers NSF:CCF:1618717, NSF:CMMI:1663256 and NSF:CCF:1740796.

\bibliographystyle{gOMS}
\bibliography{multi_refs}

\newpage
\appendices

\section{Extended Numerical Results - Real Datasets - Logistic Regression}
\label{sec:extnumres}

In this section, we present further numerical results on binary classification logistic regression problems, on the datasets listed in Table \ref{tbl:alldatasets}, in both the multi-batch and fault-tolerant settings. Note, that some of the datasets are too small, and thus, there is no reason to run them on a distributed platform; however, we include them as they are part of the
standard benchmarking datasets.

\begin{table}[htp]
\centering
\caption{Datasets and basic statistics. All datasets are available at \url{https://www.csie.ntu.edu.tw/~cjlin/libsvmtools/datasets/binary.html}.\label{tbl:alldatasets}}
\begin{tabular}{lrrrr}
\toprule
\multicolumn{1}{l}{\textbf{Dataset}} & \textbf{$\pmb{n}$} & \textbf{$\pmb{d}$} & \textbf{Size (MB)} & \textbf{K} \\  \midrule
\textbf{ijcnn} &  91,701 & 22 & 14 & 4 \\ \hdashline
\textbf{cov} & 581,012 & 54 & 68 & 4 \\ \hdashline
\textbf{news20} & 19,996 & 1,355,191 & 134 & 4\\ \hdashline
\textbf{rcvtest} & 677,399 & 47,236 & 1,152 & 16 \\ \hdashline
\textbf{url} & 2,396,130 & 3,231,961& 2,108&16 \\ \hdashline
\textbf{kdda}  &	8,407,752	&	20,216,830 & 2,546 & 16\\ \hdashline
\textbf{kddb}  & 19,264,097	&	29,890,095 & 4,894 & 16 \\ \hdashline
\textbf{webspam} & 350,000 & 16,609,143 & 23,866 & 16 \\ \hdashline
\textbf{splice-site} & 50,000,000 &  11,725,480  & 260,705 & 16 \\
 \bottomrule
\end{tabular} 
\end{table}

We focus on logistic regression classification; the objective function is given by 
\begin{align*}
	\min_{w \in \mathbb{R}^d} F(w) =  \frac{1}{n}\sum_{i=1}^{n}\log\left(1+e^{-y^i(w^Tx^i)}\right)
  + \frac{\sigma}{2} \|w\|^2,
\end{align*}
where $ (x^i, y^i)_{i=1}^n$  denote the training examples and $\sigma = \frac{1}{n}$ is the regularization parameter.

\subsection{Extended Numerical Results - Multi-Batch Setting}
\label{sec:extnumres_mb}
For the experiments in this section (Figures \ref{fig:ijcnn}-\ref{fig:splice}), we ran four methods:
\begin{itemize}
	\item (Robust L-BFGS) robust multi-batch L-BFGS (Algorithm \ref{alg:multi}),
	\item (L-BFGS) multi-batch L-BFGS without enforcing sample consistency; gradient differences are computed using different samples, i.e., $y_k = g_{k+1}^{S_{k+1}}-g_{k}^{S_{k}}$,
	\item (Gradient Descent) multi-batch gradient descent; obtained by setting $H_k = I$ in Algorithm \ref{alg:multi},
	\item (SGD) serial SGD; at every iteration one sample is used to compute the gradient.
\end{itemize}
In Figures \ref{fig:ijcnn}-\ref{fig:splice} we 
show the evolution of $\|\nabla F(w)\|$ for different step lengths $\alpha$, and for various batch ($\left| S\right| = r\cdot n $) and overlap ($\left| O\right| = o \cdot \left| S\right| $) sizes. Each Figure (\ref{fig:ijcnn}-\ref{fig:splice}) consists of 10 plots that illustrate the performance of the methods with the following parameters:
\begin{itemize}
	\item Top 3 plots: $\alpha=1$, $o=20\%$ and $r=1\%,5\%,10\%$
	\item Middle 3 plots: $\alpha=0.1$, $o=20\%$ and $r=1\%,5\%,10\%$
	\item Bottom 4 plots: $\alpha=1$, $r=1\%$ and $o=5\%,10\%,20\%,30\%$
\end{itemize}

\begin{figure}[H]
\centering

\includegraphics[width=\textwidth]{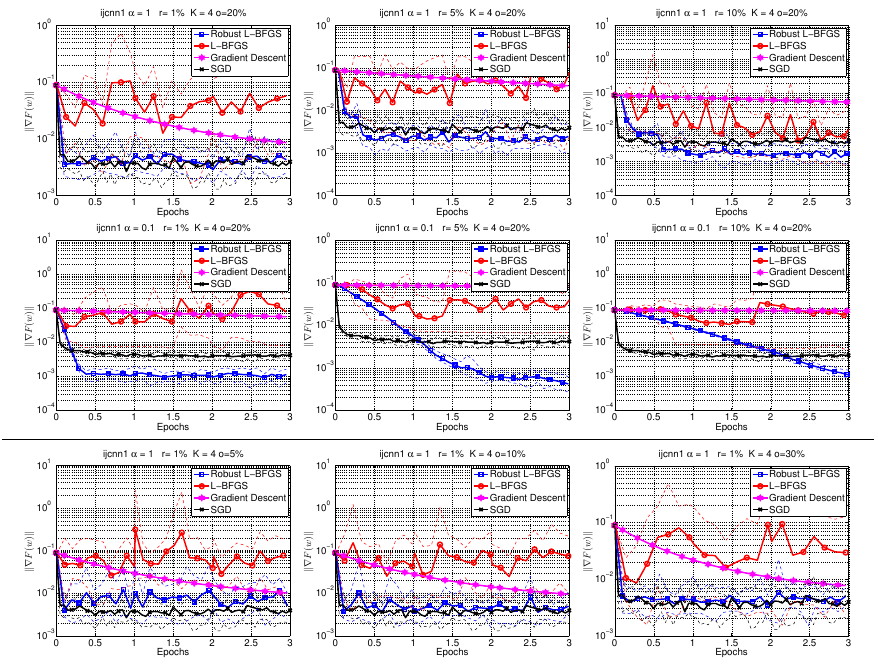}

\caption{\textbf{ijcnn1 dataset}. Comparison of Robust L-BFGS, L-BFGS (multi-batch L-BFGS without enforcing sample consistency), Gradient Descent (multi-batch Gradient method) and SGD. Top part:
we used $\alpha \in \{1, 0.1\}$,
$r\in \{1\%,  5\%,  10\%\}$ and $o=20\%$.
Bottom part: we used $\alpha=1$, $r=1\%$ and
$o\in \{5\%,  10\%, 30\%\}$. Solid lines show average performance, and dashed lines show worst and best performance, over 10 runs (per algorithm). $K=4$ MPI processes.}
\label{fig:ijcnn}
\end{figure}

\begin{figure}[H]
\centering

\includegraphics[width=\textwidth]{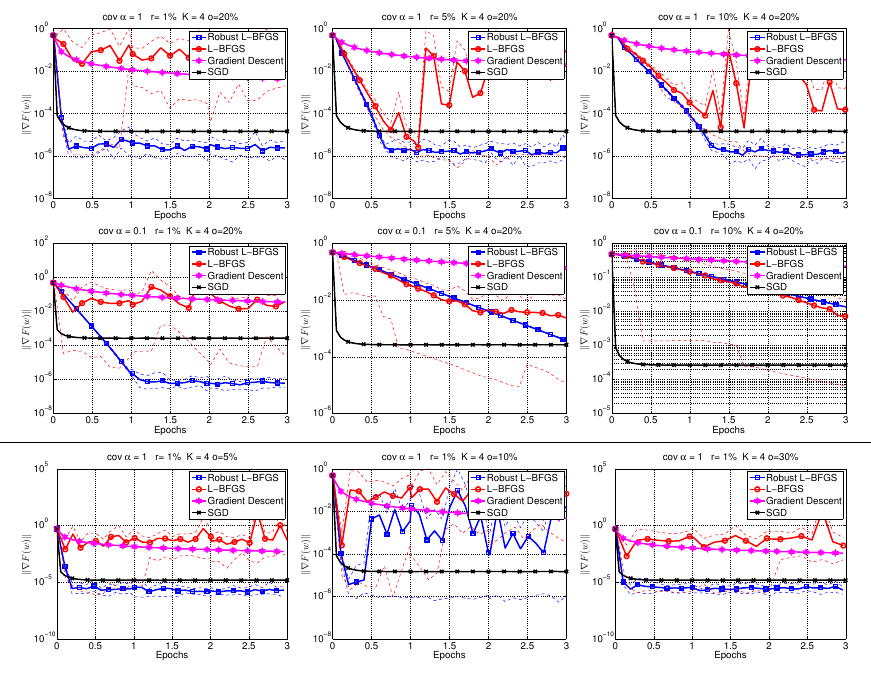}

\caption{\textbf{cov dataset}. Comparison of Robust L-BFGS, L-BFGS (multi-batch L-BFGS without enforcing sample consistency), Gradient Descent (multi-batch Gradient method) and SGD. Top part:
we used $\alpha \in \{1, 0.1\}$,
$r\in \{1\%,  5\%,  10\%\}$ and $o=20\%$.
Bottom part: we used $\alpha=1$, $r=1\%$ and
$o\in \{5\%,  10\%, 30\%\}$. Solid lines show average performance, and dashed lines show worst and best performance, over 10 runs (per algorithm). $K=4$ MPI processes.}
\end{figure}

\begin{figure}[H]
\centering

\includegraphics[width=\textwidth]{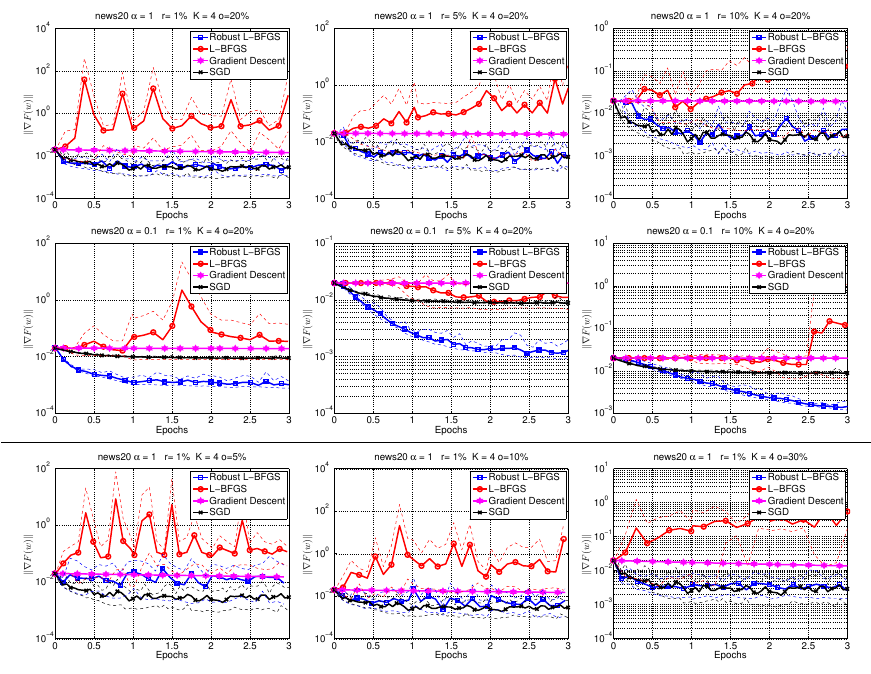}

\caption{\textbf{news20 dataset}. Comparison of Robust L-BFGS, L-BFGS (multi-batch L-BFGS without enforcing sample consistency), Gradient Descent (multi-batch Gradient method) and SGD. Top part:
we used $\alpha \in \{1, 0.1\}$,
$r\in \{1\%,  5\%,  10\%\}$ and $o=20\%$.
Bottom part: we used $\alpha=1$, $r=1\%$ and
$o\in \{5\%,  10\%, 30\%\}$. Solid lines show average performance, and dashed lines show worst and best performance, over 10 runs (per algorithm). $K=4$ MPI processes.}
\end{figure}

\begin{figure}[H]
\centering

\includegraphics[width=\textwidth]{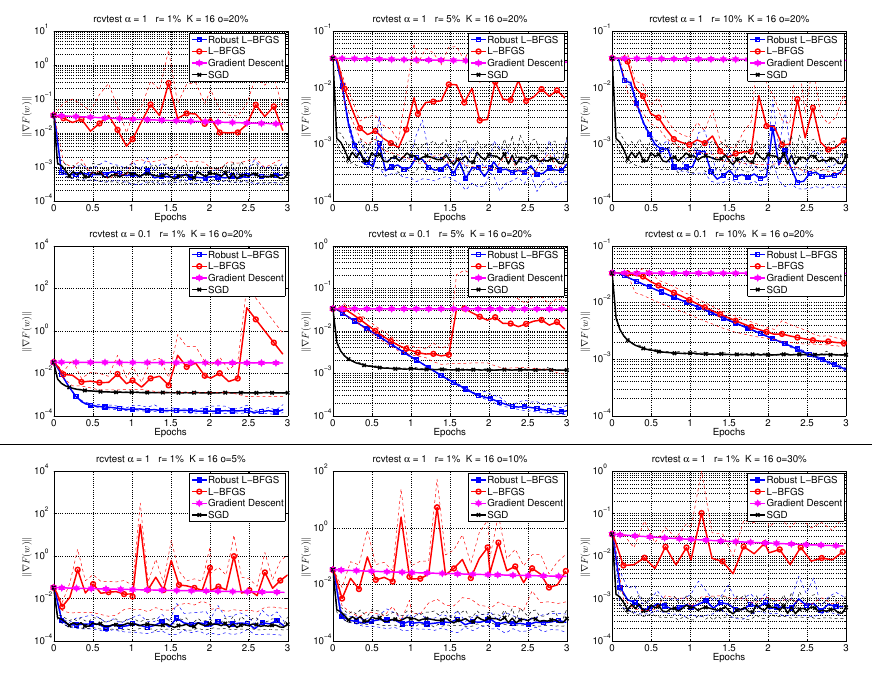}

\caption{\textbf{rcvtest dataset}. Comparison of Robust L-BFGS, L-BFGS (multi-batch L-BFGS without enforcing sample consistency), Gradient Descent (multi-batch Gradient method) and SGD. Top part:
we used $\alpha \in \{1, 0.1\}$,
$r\in \{1\%,  5\%,  10\%\}$ and $o=20\%$.
Bottom part: we used $\alpha=1$, $r=1\%$ and
$o\in \{5\%,  10\%, 30\%\}$. Solid lines show average performance, and dashed lines show worst and best performance, over 10 runs (per algorithm). $K=16$ MPI processes.}
\end{figure}

\begin{figure}[H]
\centering

\includegraphics[width=\textwidth]{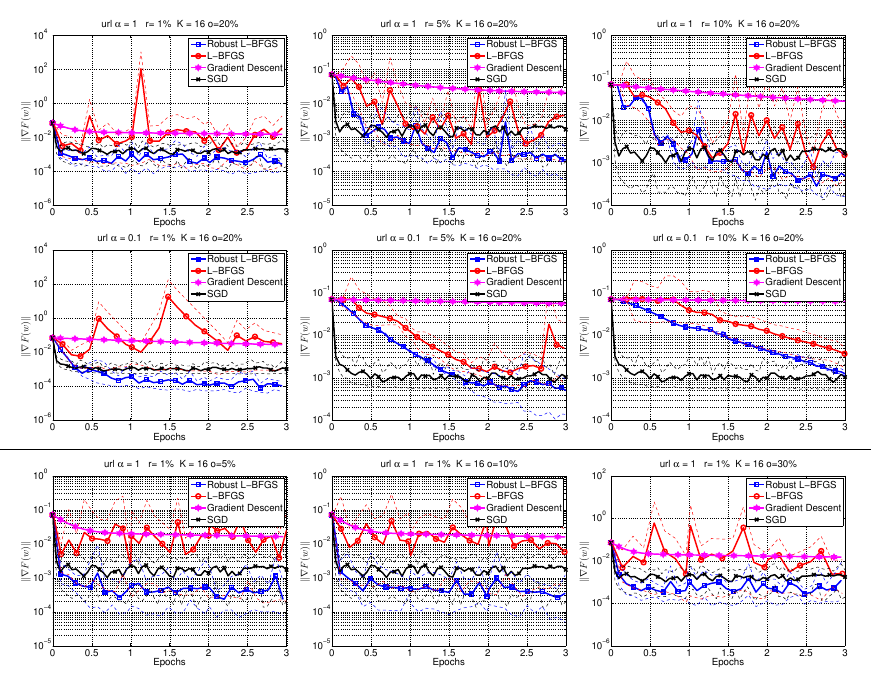}

\caption{\textbf{url dataset}. Comparison of Robust L-BFGS, L-BFGS (multi-batch L-BFGS without enforcing sample consistency), Gradient Descent (multi-batch Gradient method) and SGD. Top part:
we used $\alpha \in \{1, 0.1\}$,
$r\in \{1\%,  5\%,  10\%\}$ and $o=20\%$.
Bottom part: we used $\alpha=1$, $r=1\%$ and
$o\in \{5\%,  10\%, 30\%\}$. Solid lines show average performance, and dashed lines show worst and best performance, over 10 runs (per algorithm). $K=16$ MPI processes.}
\end{figure}

\begin{figure}[H]
\centering

\includegraphics[width=\textwidth]{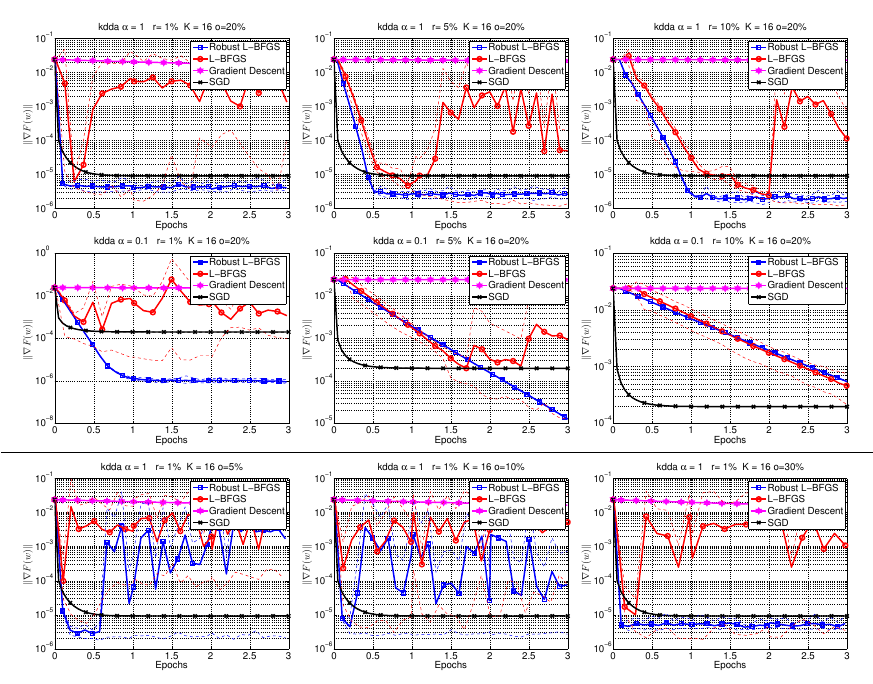}

\caption{\textbf{kdda dataset}. Comparison of Robust L-BFGS, L-BFGS (multi-batch L-BFGS without enforcing sample consistency), Gradient Descent (multi-batch Gradient method) and SGD. Top part:
we used $\alpha \in \{1, 0.1\}$,
$r\in \{1\%,  5\%,  10\%\}$ and $o=20\%$.
Bottom part: we used $\alpha=1$, $r=1\%$ and
$o\in \{5\%,  10\%, 30\%\}$. Solid lines show average performance, and dashed lines show worst and best performance, over 10 runs (per algorithm). $K=16$ MPI processes.}
\end{figure}

\begin{figure}[H]
\centering

\includegraphics[width=\textwidth]{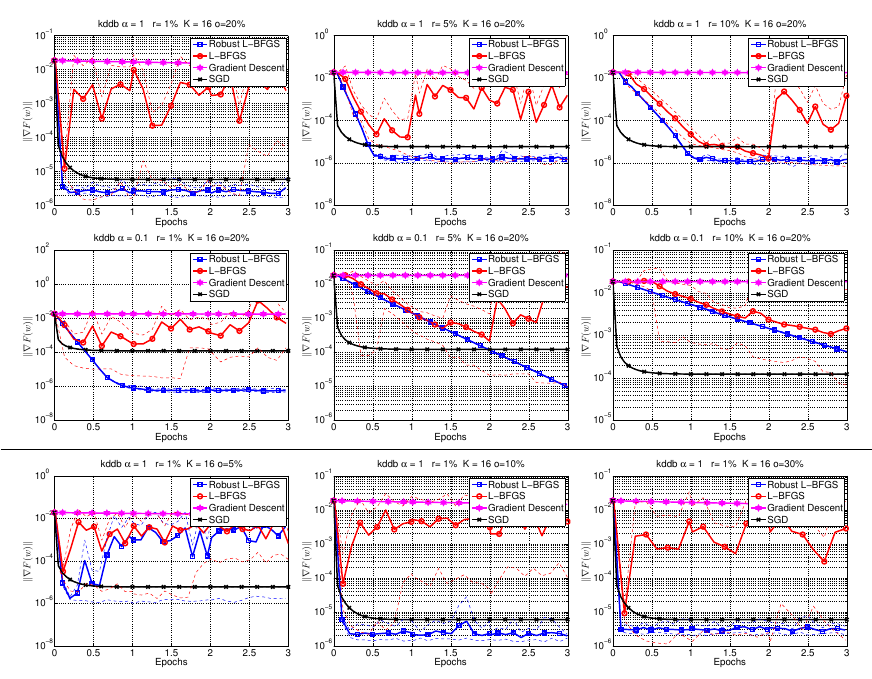}

\caption{\textbf{kddb dataset}. Comparison of Robust L-BFGS, L-BFGS (multi-batch L-BFGS without enforcing sample consistency), Gradient Descent (multi-batch Gradient method) and SGD. Top part:
we used $\alpha \in \{1, 0.1\}$,
$r\in \{1\%,  5\%,  10\%\}$ and $o=20\%$.
Bottom part: we used $\alpha=1$, $r=1\%$ and
$o\in \{5\%,  10\%, 30\%\}$. Solid lines show average performance, and dashed lines show worst and best performance, over 10 runs (per algorithm). $K=16$ MPI processes.}
\end{figure}

\begin{figure}[H]
\centering

\includegraphics[width=\textwidth]{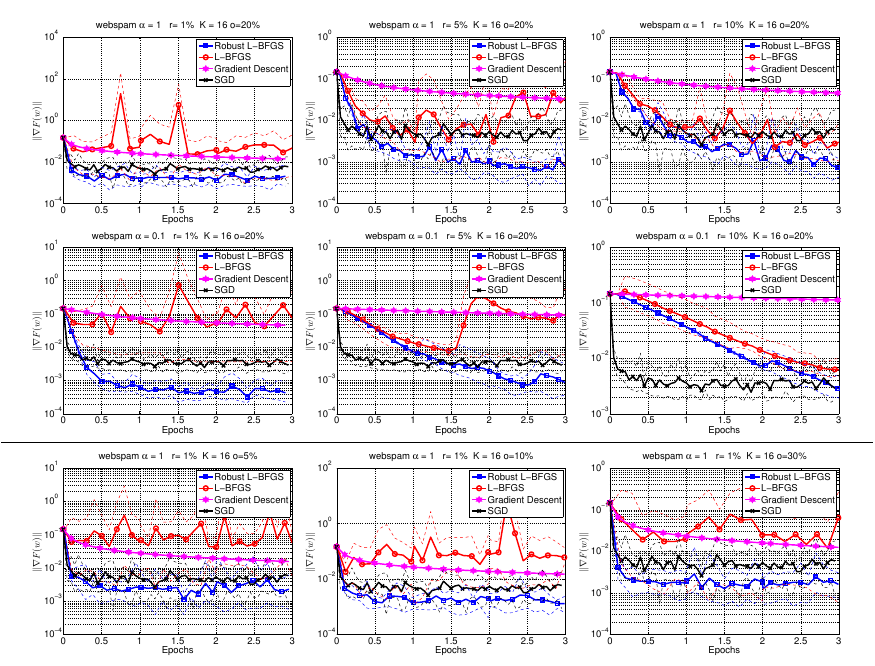}

\caption{\textbf{webspam dataset}. Comparison of Robust L-BFGS, L-BFGS (multi-batch L-BFGS without enforcing sample consistency), Gradient Descent (multi-batch Gradient method) and SGD. Top part:
we used $\alpha \in \{1, 0.1\}$,
$r\in \{1\%,  5\%,  10\%\}$ and $o=20\%$.
Bottom part: we used $\alpha=1$, $r=1\%$ and
$o\in \{5\%,  10\%, 30\%\}$. Solid lines show average performance, and dashed lines show worst and best performance, over 10 runs (per algorithm). $K=16$ MPI processes.}
\end{figure}

\begin{figure}[H]
\centering

\includegraphics[width=\textwidth]{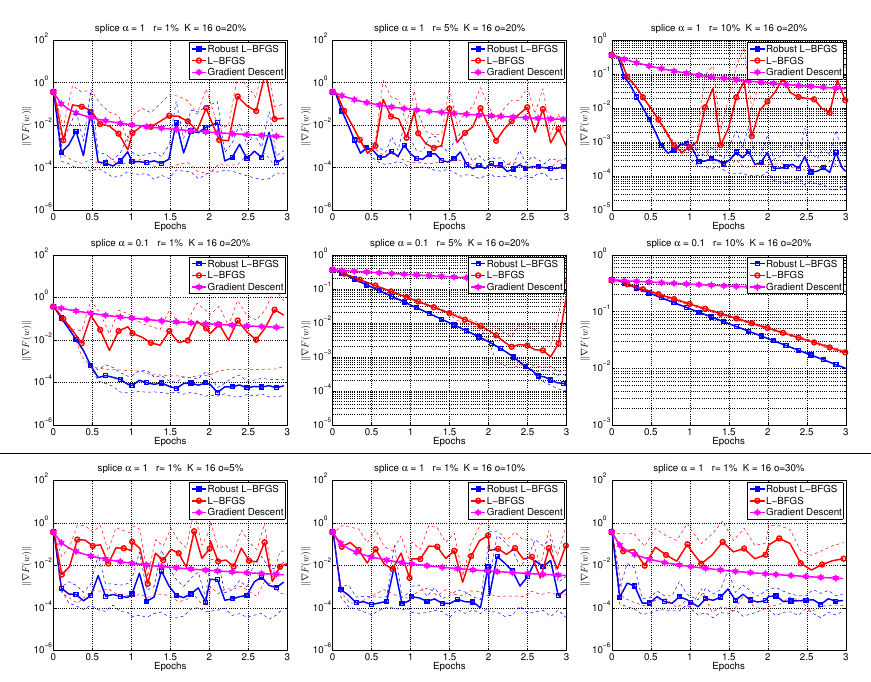}

\caption{\textbf{splice-cite dataset}. Comparison of Robust L-BFGS, L-BFGS (multi-batch L-BFGS without enforcing sample consistency), Gradient Descent (multi-batch Gradient method) and SGD. Top part:
we used $\alpha \in \{1, 0.1\}$,
$r\in \{1\%,  5\%,  10\%\}$ and $o=20\%$.
Bottom part: we used $\alpha=1$, $r=1\%$ and
$o\in \{5\%,  10\%, 30\%\}$. Solid lines show average performance, and dashed lines show worst and best performance, over 10 runs (per algorithm). $K=16$ MPI processes. (No Serial SGD experiments due to memory limitations of our cluster.)
}
\label{fig:splice}
\end{figure}

\newpage

\subsection{Extended Numerical Results - Fault-Tolerant Setting}
\label{sec:extnumres_ft}

If we run a distributed algorithm, for example on a shared computer cluster, then we may experience  delays. Such delays can be caused by other processes running on the same compute node, node failures and/or for other reasons. As a result, given a computational (time) budget, these delays may cause nodes to fail to return a value. To illustrate this behavior, and to motivate the robust fault-tolerant L-BFGS method, we run a simple benchmark MPI code on two different environments:
\begin{itemize}
\item {\bf Amazon EC2} -- Amazon EC2 is a cloud system provided by Amazon. It is expected that if load balancing is done properly, the execution time will have small noise; however, the network and communication can still be an issue. (4 MPI processes)
\item {\bf Shared Cluster} -- On our shared cluster, multiple jobs run on each node, with some jobs being more demanding than others. Even though each node has 16 cores, the amount of resources each job can utilize changes over time. In terms of communication, we have a GigaBit network. (11 MPI processes, running on 11 nodes)
\end{itemize}
We run a simple code on the cloud/cluster, with MPI communication. We generate two matrices $A,B \in  R^{n \times n}$, then  synchronize all MPI processes and compute $C=A\cdot B$ using the GSL C-BLAS library. The time is measured and recorded as computational time. After the matrix product is computed, the result is sent to the master/root node using asynchronous communication, and the time required is recorded. The process is repeated 3000 times. 

\begin{figure}[h!]
\centering

\includegraphics[width=0.65\textwidth]{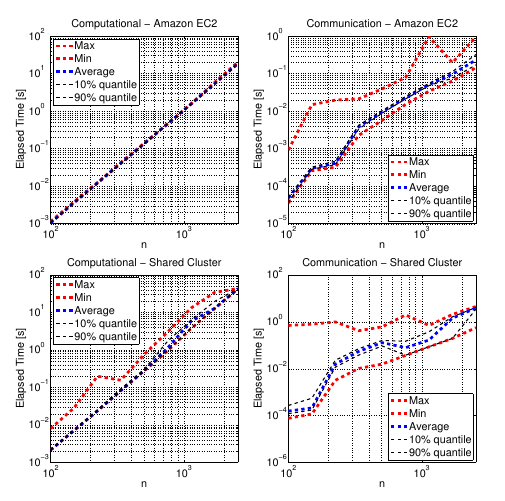}

\caption{Distribution of Computation and Communication Time for  Amazon EC2 and Shared Cluster.
Figures show worst and best time, average time and 10\% and 90\% quantiles.
Amazon Cloud EC: 4 MPI processes; Shared Cluster: 11 MPI processes.}
\label{fig.motivation}

\end{figure}

The results of the experiment described above are captured in Figure \ref{fig.motivation}. As expected, on the Amazon EC2 cloud, the matrix-matrix multiplication takes roughly the same time for all replications and the noise in communication is relatively small. In this example the cost of communication is negligible when compared to the cost of computation. On our shared cluster, one cannot guarantee that all resources are exclusively used for a specific process, and thus, the computation and communication time is considerably more stochastic and unbalanced. In some cases, the difference between the minimum and maximum computation and communication times vary by an order of magnitude or more. Hence, on such a platform a fault-tolerant algorithm that only uses information from nodes that return an update within a preallocated budget is a natural choice.

In Figures  \ref{fig:ft:rcv}-\ref{fig:ft:kddb} we present a comparison of the proposed robust multi-batch L-BFGS method and the multi-batch L-BFGS method that does not enforce sample consistency (L-BFGS). In these experiments, $p$ denotes the probability that a single node (MPI process) will not return a gradient evaluated on local data within a given time budget. We illustrate the performance of the methods for $\alpha=0.1$ and $p\in \{0.1, 0.2, 0.3, 0.4, 0.5\}$. We observe that the robust implementation is not affected much by the failure probability $p$.

\begin{figure}[h!]
\centering

\includegraphics[width=\textwidth]{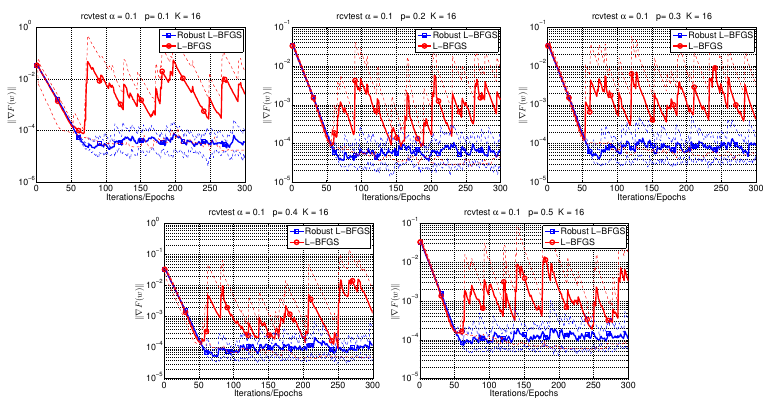}

\caption{\textbf{rcvtest dataset}. Comparison of Robust L-BFGS and L-BFGS in the presence of faults.
We used $\alpha=0.1$ and $p\in \{0.1, 0.2, 0.3, 0.4, 0.5\}$. Solid lines show average performance, and dashed lines show worst and best performance, over 10 runs (per algorithm). $K=16$ MPI processes.}
\label{fig:ft:rcv}
\end{figure}

\begin{figure}[h!]
\centering

\includegraphics[width=\textwidth]{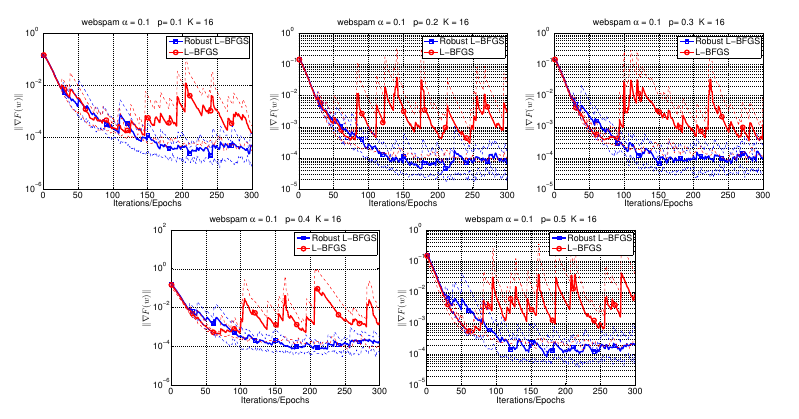}

\caption{\textbf{webspam dataset}. Comparison of Robust L-BFGS and L-BFGS in the presence of faults.
We used $\alpha=0.1$ and $p\in \{0.1, 0.2, 0.3, 0.4, 0.5\}$. Solid lines show average performance, and dashed lines show worst and best performance, over 10 runs (per algorithm). $K=16$ MPI processes.
}\label{fig:ft:webspam}
\end{figure}

\begin{figure}[h!]
\centering

\includegraphics[width=\textwidth]{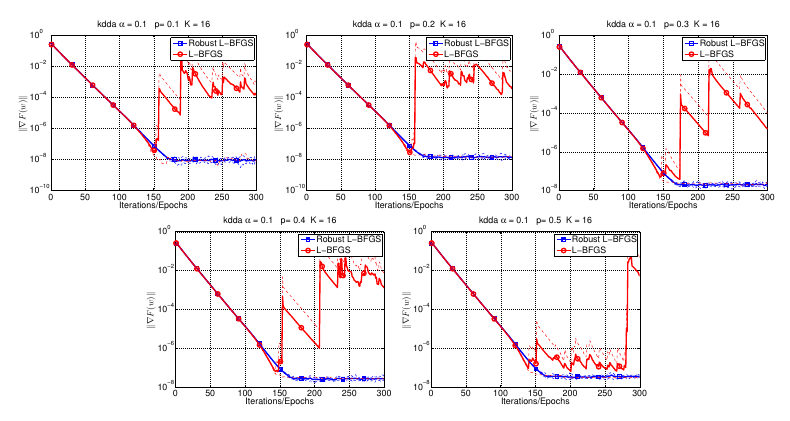}

\caption{\textbf{kdda dataset}. Comparison of Robust L-BFGS and L-BFGS in the presence of faults.
We used $\alpha=0.1$ and $p\in \{0.1, 0.2, 0.3, 0.4, 0.5\}$. Solid lines show average performance, and dashed lines show worst and best performance, over 10 runs (per algorithm). $K=16$ MPI processes.
}\label{fig:ft:kdda}
\end{figure}

\begin{figure}[h!]
\centering

\includegraphics[width=\textwidth]{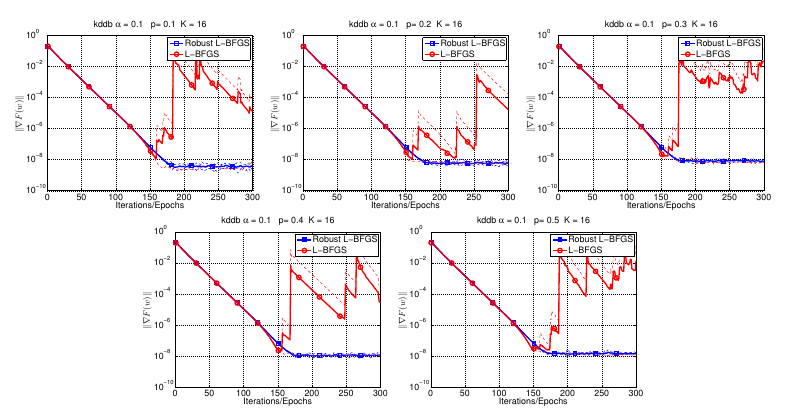}

\caption{\textbf{kddb dataset}. Comparison of Robust L-BFGS and L-BFGS in the presence of faults.
We used $\alpha=0.1$ and $p\in \{0.1, 0.2, 0.3, 0.4, 0.5\}$. Solid lines show average performance, and dashed lines show worst and best performance, over 10 runs (per algorithm). $K=16$ MPI processes.
}\label{fig:ft:kddb}
\end{figure}

\begin{figure}[h!]
\centering

\includegraphics[width=\textwidth]{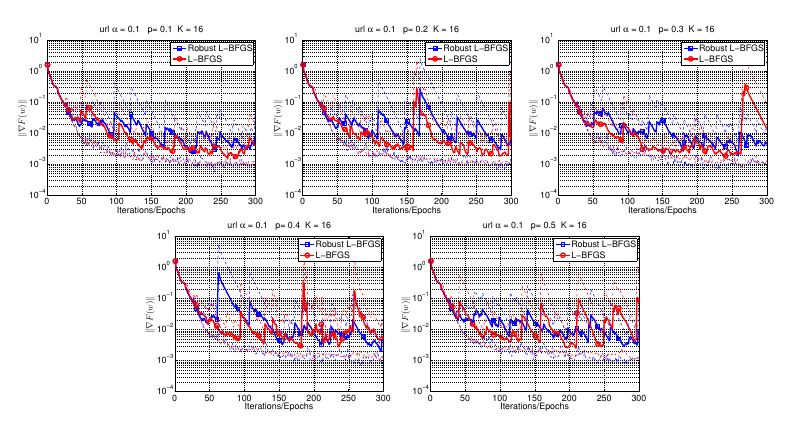}

\caption{\textbf{url dataset}. Comparison of Robust L-BFGS and L-BFGS in the presence of faults.
We used $\alpha=0.1$ and $p\in \{0.1, 0.2, 0.3, 0.4, 0.5\}$. Solid lines show average performance, and dashed lines show worst and best performance, over 10 runs (per algorithm). $K=16$ MPI processes.
}\label{fig:ft:url}
\end{figure}

\clearpage

%%%%%%%%%%%%%%%%
\section{Extended Numerical Results - Neural Networks}
\label{sec:extnumres_nn}

In this section, we present complete numerical results for the Neural Network training tasks on the MNIST\footnote{Available at: \url{http://yann.lecun.com/exdb/mnist/}.} and CIFAR10/CIFAR100\footnote{Available at: \url{https://www.cs.toronto.edu/~kriz/cifar.html}.} datasets. More specifically, we show the training and testing accuracy of the methods for all batch sizes ($| S | \in \{50, 100, 200, 500, 1000, 2000, 4000 \}$). The tasks are summarized in Table \ref{tbl:NN_app}, and the details are explained below.

\begin{table}[h!]

\caption{ Structure of Neural Networks.}
\small
\vskip5pt
\label{tbl:NN_app}
\centering
\begin{tabular}{lccrc}
\toprule
\textbf{Network} &
 \textbf{Type} &
 \textbf{\# of layers} & 
  \textbf{$d$} & \textbf{Ref.}  \\  \midrule

\textbf{MNIST MLC} & fully connected (FC)& 1& 7.8k & \cite{lecun1998gradient}
 \\ \hdashline
\textbf{MNIST DNN (SoftPlus)} & conv+FC  & 4& 1.1M &  \cite{PyTorchExamples}
\\ \hdashline
\textbf{MNIST DNN (ReLU)} & conv+FC & 4& 1.1M & \cite{PyTorchExamples} 
 \\ \hdashline
\textbf{CIFAR10 LeNet} & conv+FC & 5& 62.0k &  \cite{lecun1998gradient}
\\ \hdashline
\textbf{CIFAR10 VGG11} &
conv+batchNorm+FC
 & 29& 9.2M & \cite{simonyan2014very}
\\ \hdashline
\textbf{CIFAR100 VGG11} & conv+batchNorm+FC& 29& 9.2M & \cite{simonyan2014very}
\\
 \bottomrule
{\small MLC =  Multiclass Linear Classifier}
\end{tabular}
\end{table}

\begin{itemize}
	\item \textbf{MNIST Convex:} This problem was a convex problem. It is a simple neural network with no hidden layers and soft-max cross-entropy loss function.
	\item \textbf{MNIST DNN (SoftPlus):}
	This network consists of two convolutional layers and two fully connected layer. The first convolutional layer has $32$ filters for each $5 \times 5$ patch, and the second convolutional layer has $64$ filters for each $5 \times 5$ patch. Every convolutional layer is followed by a $2 \times 2 $ max pooling layer. The first fully connected layer has 1024 neurons and the second fully connected layer produces a 10 dimensional output. We used SoftPlus activation functions.	
	The loss function is soft-max cross-entropy loss.
	\item \textbf{MNIST DNN (ReLU):}
	Same as above, but with ReLU activation functions.
	\item \textbf{CIFAR10 LeNet:}
	This network consists of two convolutional layers and three fully connected layer. The first convolutional layer has $6$ filters for each $5 \times 5$ patch, and the second convolutional layer has $16$ filters for each $5 \times 5$ patch. Every convolutional layer is followed by a $2 \times 2 $ max pooling layer. The first fully connected layer has 120 neurons and the second fully connected layer
has 84 neurons. The last fully connected layer produces a 10 dimensional output. The activation functions used is ReLU  and the loss function used is soft-max cross-entropy loss. 
	\item \textbf{CIFAR10 VGG11 and CIFAR100 VGG11:}
This is standard VGG11 network which contains 8 convolution layers, each followed by batch-normalization and ReLU activation functions. There are also 5 max-pooling layers and one average pooling layer. The output of desired dimension (10 or 100) is achieve by fully connected layer.  	
	The loss function is soft-max cross-entropy loss. 
\end{itemize}
 
For the experiments in this section (Figures \ref{fig:nn:app1} and \ref{fig:nn:metric2}), we ran the following methods:
\begin{itemize}
	\item (LFBGS) multi-batch L-BFGS \textcolor{blue}{with the cautious strategy}: Algorithm \ref{alg:multi} with standard initial scaling ($\gamma_k I$, where $\gamma_k = \frac{s_{k-1}^T y_{k-1}}{y_{k-1}^T y_{k-1}}$) or $\alpha I$ initial scaling,
	\item (Adam) \cite{kingma2014adam},
	\item (SGD) \cite{robbins1951stochastic}.
\end{itemize}

Figures \ref{fig:nn:app1} and \ref{fig:nn:app6}  
show the evolution of training and testing accuracy for different batch sizes for the different neural network training tasks. Figures \ref{fig:nn:metric1} and \ref{fig:nn:metric2} show the two diagnostic metrics for different batch sizes.

\begin{figure}[H]
\centering

\includegraphics[width=\textwidth]{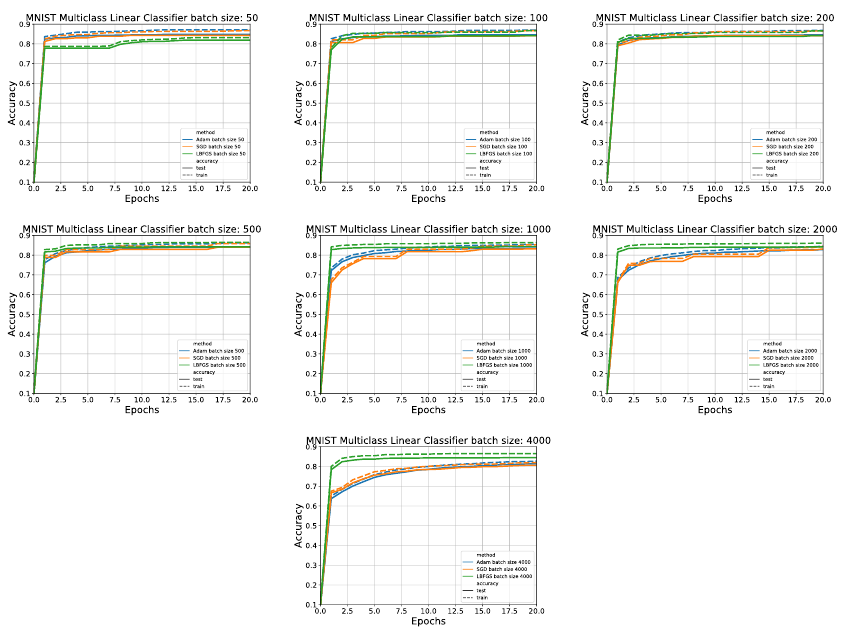}

\caption{\textbf{MNIST Multiclass Linear Classifier.} Comparison of multi-batch L-BFGS (LBFGS), Adam and SGD. Evolution of the training and testing accuracy for different batch sizes ($| S | \in \{50, 100, 200, 500, 1000, 2000, 4000 \}$). \label{fig:nn:app1}}
\end{figure}

\begin{figure}
\centering

\includegraphics[width=\textwidth]{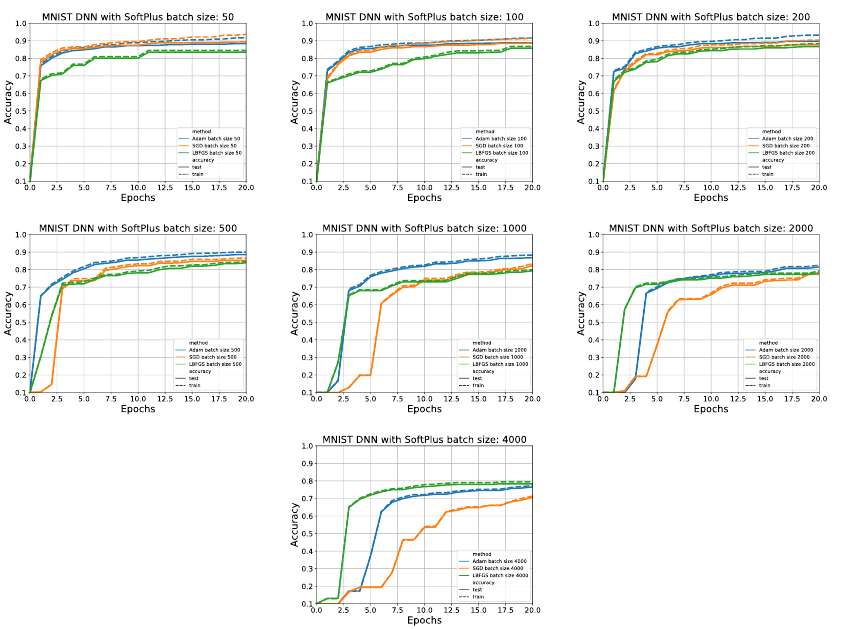}

\caption{\textbf{MNIST DNN (SoftPlus).} Comparison of multi-batch L-BFGS (LBFGS), Adam and SGD. Evolution of the training and testing accuracy for different batch sizes ($| S | \in \{50, 100, 200, 500, 1000, 2000, 4000 \}$). \label{fig:nn:app2}}
\end{figure}

\begin{figure}
\centering

\includegraphics[width=\textwidth]{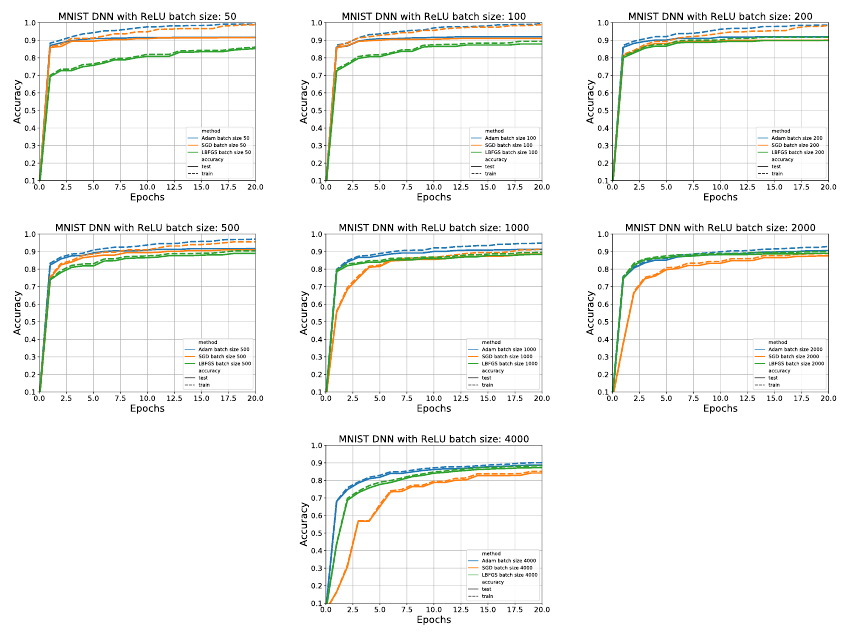}

\caption{\textbf{MNIST DNN (ReLU).} Comparison of multi-batch L-BFGS (LBFGS), Adam and SGD. Evolution of the training and testing accuracy for different batch sizes ($| S | \in \{50, 100, 200, 500, 1000, 2000, 4000 \}$). \label{fig:nn:app3}}
\end{figure}

\begin{figure}
\centering

\includegraphics[width=\textwidth]{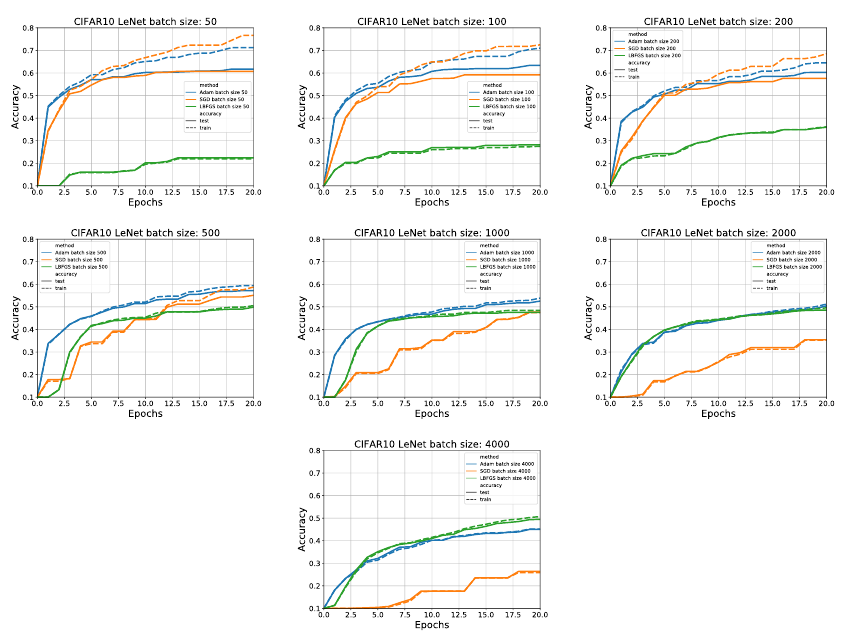}

\caption{\textbf{CIFAR10 LeNet.} Comparison of multi-batch L-BFGS (LBFGS), Adam and SGD. Evolution of the training and testing accuracy for different batch sizes ($| S | \in \{50, 100, 200, 500, 1000, 2000, 4000 \}$). \label{fig:nn:app4}}
\end{figure}

\begin{figure}
\centering

\includegraphics[width=\textwidth]{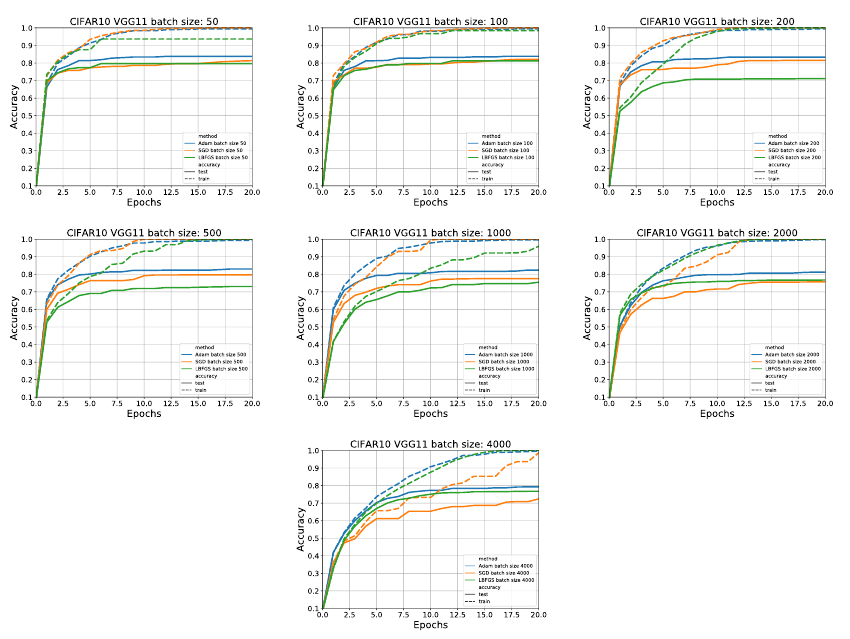}

\caption{\textbf{CIFAR10 VGG11.} Comparison of multi-batch L-BFGS (LBFGS), Adam and SGD. Evolution of the training and testing accuracy for different batch sizes ($| S | \in \{50, 100, 200, 500, 1000, 2000, 4000 \}$). \label{fig:nn:app5}}
\end{figure}

\begin{figure}
\centering

\includegraphics[width=\textwidth]{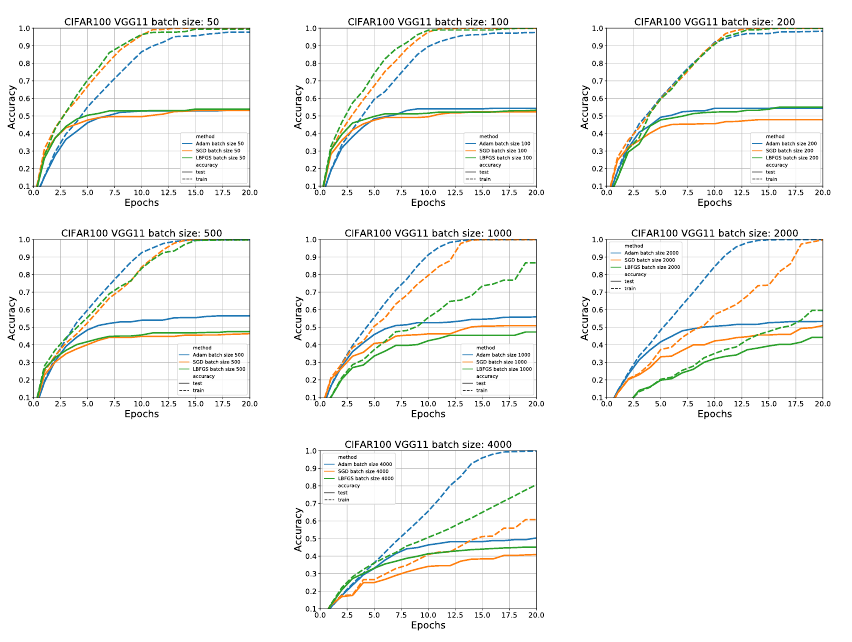}

\caption{\textbf{CIFAR100 VGG11.} Comparison of multi-batch L-BFGS (LBFGS), Adam and SGD. Evolution of the training and testing accuracy for different batch sizes ($| S | \in \{50, 100, 200, 500, 1000, 2000, 4000 \}$). \label{fig:nn:app6}}
\end{figure}

\clearpage

\begin{figure}
\centering

\includegraphics[width=\textwidth]{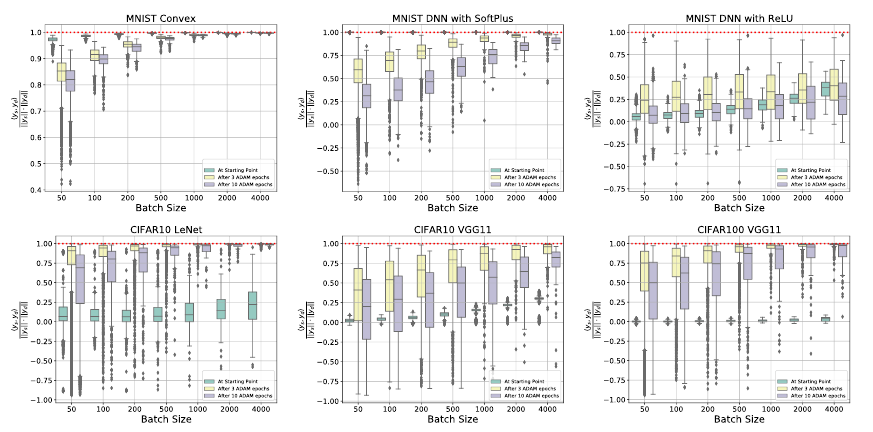}

\caption{Multi-batch L-BFGS diagnostic metric for different batch sizes $| S | \in \{50, 100, 200, 500, 1000, 2000, 4000 \}$. Angle between the true gradient curvature vector $y_d$ and subsampled gradient curvature vector $y_s$ ($\frac{\langle y_s,y_d \rangle}{\| y_s\| \| y_d \|}$). Top right: MNIST Multiclass Linear Classifier; Top middle: MNIST DNN (SoftPlus); Top Right: MNIST DNN (ReLU); Bottom left: CIFAR10 LeNet; Bottom middle: CIFAR10 VGG11; Bottom right: CIFAR100 VGG11. \label{fig:nn:metric1}}
\end{figure}

\begin{figure}
\centering

\includegraphics[width=\textwidth]{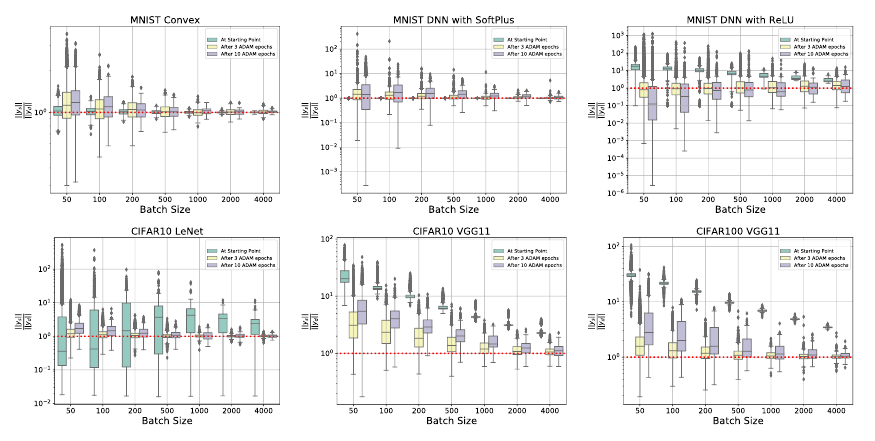}

\caption{Multi-batch L-BFGS diagnostic metric for different batch sizes $| S | \in \{50, 100, 200, 500, 1000, 2000, 4000 \}$. Ratio of subsampled gradient curvature vector to true gradient curvature vector ($\frac{y_s}{y_d}$). Top right: MNIST Multiclass Linear Classifier; Top middle: MNIST DNN (SoftPlus); Top Right: MNIST DNN (ReLU); Bottom left: CIFAR10 LeNet; Bottom middle: CIFAR10 VGG11; Bottom right: CIFAR100 VGG11.  \label{fig:nn:metric2}}
\end{figure}

\end{document}